\documentclass[a4paper]{amsart}
\usepackage{amsmath}
\usepackage{amssymb}
\usepackage{amsfonts}
\usepackage{pxfonts}
\usepackage[OT2,T1]{fontenc}

\usepackage{pict2e}

\usepackage[
textwidth=13cm, 
textheight=21cm,
hmarginratio=1:1,
vmarginratio=1:1]{geometry}

\newtheorem{thm}{Theorem}[subsection]
\newtheorem{cor}[thm]{Corollary}
\newtheorem{lem}[thm]{Lemma}
\newtheorem{prop}[thm]{Proposition}
\theoremstyle{definition}
\newtheorem{defn}[thm]{Definition}
\newtheorem{conj}[thm]{Conjecture}

\theoremstyle{remark}
\newtheorem{rem}[thm]{Remark}

\numberwithin{equation}{subsection}
\numberwithin{figure}{subsection}

\newcommand{\diff}{\mathrm{d}}
\newcommand{\C}{{\mathbb C}}
\newcommand{\R}{{\mathbb R}}
\newcommand{\D}{{\mathbb D}}
\newcommand{\expect}{{\mathbb E}}
\newcommand{\Te}{{\mathbb T}}
\newcommand{\Z}{{\mathbb Z}}

\newcommand{\imag}{\mathrm{i}}
\newcommand{\e}{\mathrm{e}}
\newcommand{\hDelta}{\varDelta}

\newcommand{\pv}{\mathrm{pv}}

\newcommand{\Pop}{{\mathbf P}}

\newcommand{\Mop}{{\mathbf M}}

\newcommand{\Ordo}{\mathrm{O}}

\newcommand{\classS}{\mathcal{S}}
\newcommand{\cauchy}{\mathbf{C}}
\newcommand{\Sop}{\mathbf{S}}

\newcommand{\hfun}{h}

\DeclareMathOperator{\re}{Re}

\begin{document}

\title{Bloch functions and asymptotic tail variance}

%\author{Alexander Borichev}

%\address{Borichev: Laboratoire d'analyse, topologie, probabilit\'es\\
%CMI, Aix-Marseille Universit\'e\\
%39, rue Fr\'ed\'eric Joliot Curie\\
%F--13453 Marseille CEDEX 13\\
%FRANCE}
%
%\email{borichev@cmi.univ-mrs.fr}

\author{Haakan Hedenmalm}
\address{
Hedenmalm: Department of Mathematics\\
KTH Royal Institute of Technology\\
S--10044 Stockholm\\
SWEDEN
\\
second affiliation: Chebyshev Laboratory\\
Vasilievsky Island, St-Petersburg 199178\\
RUSSIA}

\email{haakanh@math.kth.se}

 
\thanks{This research was 
%results of the author in Sections 1-8 were supported by 
%Vetenskapsr\aa{}det (VR) dnr 2012-3122, while the results in 
%Sections 9-10 were
supported by RNF grant 14-41-00010}

\begin{abstract} 
Let $\Pop$ denote the Bergman projection on the unit disk $\D$,
\[
\Pop \mu(z):=\int_\D\frac{\mu(w)}{(1-z\bar w)^2}\,\diff A(w),\qquad z\in\D,
\]
where $\diff A$ is normalized area measure. We prove that if $|\mu(z)|\le1$
on $\D$, then the integral 
\[
I_\mu(a,r):=\int_{0}^{2\pi}\exp\Bigg\{a\frac{r^4|\Pop\mu(r\e^{\imag\theta})|^2}
{\log\frac{1}{1-r^2}}\Bigg\}\frac{\diff\theta}{2\pi},\qquad 0<r<1,
\]
has the bound $I_\mu(a,r)\le C(a):=10(1-a)^{-3/2}$ for $0<a<1$, irrespective of
the choice of the function $\mu$. Moreover, for $a>1$, no such uniform bound 
is possible. We interpret the theorem in terms the \emph{asymptotic
tail variance} of such a Bergman projection $\Pop\mu$ (by the way, the 
asymptotic tail variance induces a seminorm on the Bloch space). 
This improves upon earlier work of Makarov, which covers the range 
$0<a<\frac{\pi^2}{64}=0.1542\ldots$. We then apply the theorem 
to obtain an estimate of the universal integral means spectrum for conformal
mappings with a $k$-quasiconformal extension, for $0<k<1$. The estimate
reads, for $t\in\C$ and $0<k<1$, 
\[
\mathrm{B}(k,t)\le
\begin{cases}
\frac14 k^2|t|^2(1+7k)^2,\quad \text{for}\quad\,|t|\le\frac{2}{k(1+7k)^2},
\\
 k|t|-\frac{1}{(1+7k)^2},\qquad\,\,\text{for}
\quad\,|t|\ge\frac{2}{k(1+7k)^2}, 
\end{cases}
\]  
which should be compared with the conjecture by 
%Istv\'an 
Prause and 
%Stanislav 
Smirnov to the effect that for real $t$ with $|t|\le 2/k$, we should have
$\mathrm{B}(k,t)=\frac{1}{4}k^2t^2$. 
%For large $|t|$, the estimate allows
%an improvement:
%\[
%\mathrm{B}(k,t)\le k|t|-\frac{1}{(1+7k)^2},\quad\text{for}\quad 
%|t|\ge\frac{2}{k(1+7k)^2}.
%\]
%for real $t$is conjectured to be asymptotically sharp as $k\to0$ by Prause 
%and Smirnov. 
%Furthermore, it is shown how this leads to the upper Minkowski dimension 
%bound for $k'$-quasicircles $D_M^+(k')\le 1+(k')^2+\Ordo((k')^3)$ 
%as $k'\to0^+$. 
\end{abstract}

\maketitle

%\centerline{\em In memory of Boris Korenblum}

\section{Introduction} 

\subsection{Basic notation}
\label{subsec-1.1}
We write $\R$ for the real line, $\R_+:=]0,+\infty[$ for the positive 
semi-axis, and $\C$ for the complex plane. Moreover, we write 
$\C_\infty:=\C\cup\{\infty\}$ for the extended complex plane 
(the Riemann sphere). For a complex variable $z=x+\imag y\in\C$, let 
\[
\diff s(z):=\frac{|\diff z|}{2\pi},\qquad
\diff A(z):=\frac{\diff x\diff y}{\pi},
\]
denote the normalized arc length and area measures as indicated. Moreover, we 
shall write 
\[
\varDelta_z:=\frac{1}{4}\bigg(\frac{\partial^2}{\partial x^2}+
\frac{\partial^2}{\partial y^2}\bigg)
\]
for the normalized Laplacian, and
\[
\partial_z:=\frac{1}{2}\bigg(\frac{\partial}{\partial x}-\imag
\frac{\partial}{\partial y}\bigg),\qquad
\bar\partial_z:=\frac{1}{2}\bigg(\frac{\partial}{\partial x}+\imag
\frac{\partial}{\partial y}\bigg),
\]
for the standard complex derivatives; then $\varDelta$ factors as
$\varDelta_z=\partial_z\bar\partial_z$. Often we will drop the subscript 
for these differential operators when it is obvious from the context with
respect to which variable they apply.
We let $\C$ denote the complex plane, $\D$ the open unit disk, 
$\Te:=\partial\D$ the unit circle, and $\D_e$ the exterior disk:
\[
\D:=\{z\in\C:\,\,|z|<1\},\qquad \D_e:=\{z\in\C_\infty:\,\,|z|>1\}.
\]
More generally, we write
\[
\D(z_0,r):=\{z\in\C:\,|z-z_0|<r\}
\]
for the open disk of radius $r$ centered at $z_0$.

\subsection{Dual action notation}

We will find it useful to introduce the sesquilinear forms 
$\langle\cdot,\cdot \rangle_\Te$ and $\langle\cdot,\cdot \rangle_\D$,
as given by
\[
\langle f,g \rangle_\Te:=\int_\Te f(z)\bar g(z)\diff s(z),\qquad
\langle f,g\rangle_\D:=\int_\D f(z)\bar g(z)\diff A(z),
\] 
where, in the first case, $f\bar g\in L^1(\Te)$ is required, and in the
second, we need that  $f\bar g\in L^1(\D)$.
%indication of the differentiation variable $z$

\subsection{The Bergman projection of bounded functions and the 
main result}

For a function $f\in L^1(\D)$, its \emph{Bergman projection} is the function 
$\Pop f$, as defined by 
\begin{equation}
\Pop f(z):=\int_\D\frac{f(w)}{(1-z\bar w)^2}\,\diff A(w),\qquad z\in\D.
\label{eq-Bergproj1.1}
\end{equation}
The function $\Pop f$ is then holomorphic in the disk $\D$. 
We shall be concerned with the boundary behavior of holomorphic functions of 
the type $\Pop\mu$, where $\mu\in L^\infty(\D)$, in which case $\Pop\mu$
is in the \emph{Bloch space} (see Subsection \ref{subsec-BlochBloch}). 
More precisely, we shall obtain the following result.

\begin{thm}
Suppose $g=\Pop\mu$, where $\mu\in L^\infty(\D)$, and $\|\mu\|_{L^\infty(\D)}\le1$.
% Suppose that $0<a<1$.

\noindent{\rm (a)} If $0<a<1$, we then have the estimate
\[
\int_\Te\exp\Bigg\{a\frac{r^4|g(r\zeta)|^2}
{\log\frac{1}{1-r^2}}\Bigg\}\diff s(\zeta)\le C(a),\qquad 0<r<1,
\]
where $C(a)=10(1-a)^{-3/2}$.

\noindent{\rm (b)} If $1<a<+\infty$, there exists a $\mu_0\in L^\infty(\D)$ with
$\|\mu_0\|_{L^\infty(\D)}=1$ such that with $g_0:=\Pop\mu_0$,
\[
\lim_{r\to1^-}\int_\Te\exp\Bigg\{a\frac{r^4|g_0(r\zeta)|^2}
{\log\frac{1}{1-r^2}}\Bigg\}\diff s(\zeta)=+\infty.
\]
\label{thm-main}
\end{thm}

The proof is supplied in two installments: part (a) in Corollary 
\ref{cor-strong1}, and part (b) in Corollary \ref{cor-main(b)}.

In the terminology of Section \ref{sec-notionsvar} on two notions 
of asymptotic variance (in the context of probabilistic modelling), 
the main aspects of this result may be formulated as follows:
\emph{The (uniform) asymptotic tail variance of the unit ball in 
$\Pop L^\infty(\D)$ equals $1$.}

\begin{rem}
(a) At this moment, it is not clear what happens at the critical parameter 
value $a=1$.

\noindent{(b)} For small values of $a$, $0<a<\frac{\pi^2}{64}=0.1542\ldots$, 
the same bound with a different constant $C(a)$ can be obtained from an 
estimate found by Nikolai Makarov \cite{Mak1} 
(for details, see Pommerenke's book 
\cite{Pombook}, Chapter 8, as well as Subsection \ref{subsec-Gaussmod} below). 
Later, Ba\~nuelos \cite{BM} found an independent localized approach 
involving square functions which for Bloch functions gave more or less the 
same growth estimate as the one originally found by Makarov.    
\end{rem}

\subsection{Comparison with the Dirichlet integral theorem}

In \cite{ChaMar}, Alice Chang and Donald Marshall improve upon a classical 
theorem of Arne Beurling from the 1930s (see \cite{Beu}). 
Their result is that for a positive real parameter $a$, 
\emph{there exists a uniform finite integral bound}
\[
\int_\Te\exp\big\{a|f(\zeta)|^2\big\}\diff s(\zeta)\le C(a)
\]
\emph{if and only if $0<a\le1$}, where $f$ ranges over all holomorphic 
functions $f:\D\to\C$ with $f(0)=0$ and 
\[
\int_\D|f'|^2\diff A\le1.
\]
The finiteness for $0<a<1$ was covered by Beurling's work. At the superficial 
level, this is very much reminiscent of Theorem \ref{thm-main} above.
However, we can neither derive Beurling's theorem from 
Theorem \ref{thm-main}, nor can we derive  Theorem \ref{thm-main} from the 
theorem of Chang and Marshall.
To understand this, we consider the relation
\begin{equation}
f(\zeta):=\frac{r^2\zeta^2 g(r\zeta)}{\sqrt{\log\frac{1}{1-r^2}}}.
\label{eq-relation-f:g}
\end{equation}
We 
%assume that, in addition, $f(0)=g(0)=0$, and 
observe the following:

\noindent(i) If $g=\Pop\mu$ where $\|\mu\|_{L^\infty(\D)}=1$, then the 
function $f$ extends holomorphically to a disk of radius $1/r$ and hence 
has no chance of being an arbitrary element of the unit ball of the 
Dirichlet space.

\noindent(ii) Assuming only that $g=\Pop\mu$ where $\|\mu\|_{L^\infty(\D)}=1$,
we cannot control the Dirichlet norm of $f$ uniformly as
$r$ approaches $1$. Indeed, the Dirichlet integral of $f$ is
\begin{equation}
\int_\D|f'|^2\diff A=\frac{1}{\log\frac{1}{1-r^2}}\int_{\D(0,r)}
|(z^2g)'|^2\diff A,
\label{eq-relation-f:g.DI}
\end{equation}
and since the construction of holomorphic functions with given growth 
is quite precise in \cite{Sei} (see also, e.g., \cite{HRS}), and the error
term supplied by Proposition \ref{prop-2.1.3} is small in terms of its 
boundary contribution, we may find such a function $g_1=\Pop\mu_1$ with 
$\|\mu_1\|_{L^\infty(\D)}\le1$, whose derivative grows so quickly that 
\[
\int_{\D(0,r)}|(z^2g_1)'|^2\diff A\ge\epsilon_0\frac{r^4}{1-r^2},
\]
for some absolute constant $\epsilon_0>0$. With this choice $g:=g_1$, 
the growth of the expression \eqref{eq-relation-f:g.DI} is then at least 
as quick as
\[
\epsilon_0\frac{r^4}{(1-r^2)\log\frac{1}{1-r^2}},
\]
which definitely tends to infinity as $r\to1^-$.

%\begin{rem}
%
%\end{rem}

\subsection{Applications to exponential integrability}

It might be more appropriate to compare Theorem \ref{thm-main} with the 
John-Nirenberg theorem on exponential integrability of $\rm{BMO}$ functions.
We should also have in mind the Helson-Szeg\H{o} theorem \cite{HSz}, 
which gives sharp exponential integrability for the Szeg\H{o} projection of 
a bounded function (see also Garnett's book \cite{Gar}, and e.g. Wolff's paper
\cite{Wol}). 

We first begin from the wrong end. Note that by the pointwise bound of 
Lemma \ref{lem-ptwise2}, we know that if $g=\Pop\mu$ with 
$\|\mu\|_{L^\infty(\D)}\le1$, then
\[
\frac{r^4|g(r\zeta)|^2}{\log\frac{1}{1-r^2}}\le r^2|g(r\zeta)|,
\]
so that
\begin{equation}
\int_\Te\exp\Bigg\{a\frac{r^4|g(r\zeta)|^2}{\log\frac{1}{1-r^2}}\Bigg\}
\diff s(\zeta)\le\int_\Te\exp\big\{ar^2|g(r\zeta)|\big\}\diff s(\zeta),
\label{eq-JN1}
\end{equation}
and the uniform boundedness (over $0<r<1$) of the right-hand side for small 
positive values of $a$ would be very reminiscent of the John-Nirenberg 
theorem \cite{JN}, except that we would need the dilates $g_r$ to be in 
$\mathrm{BMO}(\Te)$ uniformly, which is not true for a general function 
$g=\Pop\mu$ (indeed, it is easy to cook up a $\mu$ such that the right-hand 
side in \eqref{eq-JN1} tends to infinity as we let $r\to1^-$, for any fixed
positive $a$). 
This of course fits with the inequality in \eqref{eq-JN1}, which goes 
the wrong way if we want to derive consequences of Theorem \ref{thm-main}.
%But the inequality goes the wrong way, and it is
%likely that the right-hand side cannot be controlled uniformly as 
%$r$ approaches 
%$1$. 
To obtain an estimate that works, we instead follow Marshall \cite{Mar1}
%(see \eqref{eq-Marshall4} below, with $g(z)$ replaced by $z^2g(z)$),
who obtained the inequality (see \eqref{eq-Marshall4})
\begin{equation}
\int_\Te\big|\e^{tr^2 g(r\zeta)}\big|\diff s(\zeta)
%\\
\le(1-r^2)^{-|t|^2/(4a)}\int_\Te
\exp\Bigg\{a\frac{r^4|g(r\zeta)|^2}
{\log\frac{1}{1-r^2}}\Bigg\}\diff s(\zeta),\qquad t\in\C.
\label{eq-Marshall4.01}
\end{equation}
For $|t|>2a$, a better estimate can be obtained from a combination of 
\eqref{eq-Marshall4.01} with the pointwise bound of Lemma \ref{lem-ptwise2}
below (see Proposition \ref{prop-betterest1.001} below):
\begin{equation}
\int_\Te\big|\e^{tr^2 g(r\zeta)}\big|\diff s(\zeta)
%\\
\le(1-r^2)^{a-|t|}\int_\Te
\exp\Bigg\{a\frac{r^4|g(r\zeta)|^2}
{\log\frac{1}{1-r^2}}\Bigg\}\diff s(\zeta),\qquad |t|>2a.
\label{eq-Marshall4.015}
\end{equation}
It is well-known that $g=\Pop\mu$ is the restriction to he disk $\D$ of a
function in two-dimensional BMO. In two dimensions, the John-Nirenberg theorem
would say that $\exp(\lambda g)$ is locally in area-$L^1$, if $|\lambda|$ 
is small. In particular, $\exp(\lambda g)$ is integrable on the disk $\D$, 
and an argument involving subharmonicity and averages over disks gives that
%For comparison, the two-dimensional John-Nirenberg theorem would give 
%us that for small enough $|\lambda|$,
\begin{equation}
\int_\Te\big|\e^{\lambda g(r\zeta)}\big|\diff s(\zeta)
=\Ordo((1-r^2)^{-1}), \quad \text{as}\,\,\,r\to1^-,
\label{eq-Marshall4.02}
\end{equation}
again for small $|\lambda|$. Compared with \eqref{eq-Marshall4.02},
the estimates \eqref{eq-Marshall4.01} and \eqref{eq-Marshall4.015} are much
more precise.
%In this sense, Theorem \ref{thm-main} may be regarded as a 
%strengthening of the two-dimensional John-Nirenberg theorem which 
%applies Bergman projections of bounded functions, with a precision 
%comparable to that of the Helson-Szeg\"o
%theorem \cite{HSz}.

\subsection{The type spectrum of a Bloch function}
We need the concept of the \emph{exponential type spectrum} of the function
$\e^g$, where $g:\D\to\C$ is holomorphic. 

\begin{defn}
For a holomorphic function $g:\D\to\C$, let $\beta_g:\C\to[0,+\infty]$ 
be the function given by
\[
\beta_g(t):=\limsup_{r\to1^-}
\frac{\log\int_\Te|\e^{t g(r\zeta)}|\diff s(\zeta)}{\log\frac{1}{1-r^2}}.
\] 
We call the function $\beta_g(t)$ the \emph{exponential type spectrum} 
of the (zero-free) function $\e^g$.
\label{defn-1}
\end{defn}

We may now derive an estimate from above of the exponential 
type spectrum $\beta_g(t)$, where $g=\Pop\mu$ and $\mu\in L^\infty(\D)$, 
from the estimates \eqref{eq-Marshall4.01} and \eqref{eq-Marshall4.015}, 
together with Theorem \ref{thm-main}.

\begin{cor}
Suppose $g=\Pop\mu$, where $\mu\in L^\infty(\D)$, with $\|\mu\|_{L^\infty(\D)}\le1$.
Then 
\[\beta_g(t)\le
\begin{cases} |t|^2/4,\qquad\,|t|\le2,
\\
|t|-1,\qquad |t|\ge2.
\end{cases}
\]
\label{cor-intmeast01}
\end{cor}

The proof of Corollary \ref{cor-intmeast01} is supplied in Subsection
\ref{subsec-intmeans001}.

\subsection{Control of moments}
Makarov originally formulated his result in terms of moments; Theorem 
\ref{thm-main} implies a bound on the moments as well.

\begin{cor}
Suppose that $g=\Pop\mu$, where $\mu\in L^\infty(\D)$. 
For $0<q<+\infty$, we then have the estimate 
\[
\int_\Te|g(r\zeta)|^{q}\diff s(\zeta)\le 10(3+q)^{3/2}\|\mu\|_{L^\infty(\D)}^q
\bigg(\frac{q}{2\e}\bigg)^{q/2}
\bigg(\frac{1}{r^4}\log\frac{1}{1-r^2}\bigg)^{q/2},\qquad 0<r<1.
\]
\label{cor-moments}
\end{cor}

The proof of Corollary \ref{cor-moments} is supplied in Subsection 
\ref{subsec-moments}.
We should remark that in \cite{IvKa}, Ivrii and Kayumov show how to control
low order moments (i.e., for $0<q\le(\log\frac{1}{1-r^2})^\delta$ for small 
positive $\delta$) in terms of the uniform asymptotic variance of the unit 
ball of $\Pop L^\infty(\D)$, which is smaller than $1$, by \cite{Hed2}. 
It would be natural to combine the two estimates, using, e.g., the logarithmic 
convexity of the moments with respect to $q$.  

\subsection{Application to the universal quasiconformal 
extension spectrum}

The exponential type spectrum may be defined analogously for a holomorphic 
function $g:\D_e\to\C$ as well:
\begin{equation}
\beta_g(t):=\limsup_{R\to1^+}
\frac{\log\int_\Te|\e^{t g(R\zeta)}|\diff s(\zeta)}{\log\frac{R^2}{R^2-1}}.
\label{eq-gspect1}
\end{equation}
We recall the class $\Sigma$ of conformal mappings $\psi:\D_e\to\C_\infty$,
with asymptotics $\psi(z)=z+\Ordo(1)$ as $z\to\infty$. 
For a parameter $k$ with $0<k<1$, we denote by $\Sigma^{\langle k\rangle}$ the 
collection of all $\psi\in\Sigma$ that have a $k$-quasiconformal extension 
$\tilde\psi:\C_\infty\to\C_\infty$, 
by which we mean that $\tilde\psi$ is a homeomorphism of Sobolev class 
$W^{1,2}$ with dilatation estimate
\[
|\bar\partial_z\tilde\psi(z)|\le k|\partial_z\tilde\psi(z)|,\qquad z\in\C.
\]
%Given $\psi\in\Sigma$, we can define $\mathrm{B}_\psi(t)$, 
%for $t\in\C$, as the infimum of
%all real $\beta$ for which 
%\[
%\int_\Te \big|(\psi'(R\zeta))^t\big|\diff s(\zeta)=
%\Ordo((R-1)^{-\beta})\quad\text{as}\,\,\,
%R\to1^+.
%\]
The universal spectra 
%\mathrm{B}_\Sigma(t)=
%$\mathrm{B}(1,t)$ and 
%$\mathrm{B}_{\Sigma^{\langle k\rangle}}(t)=
$\mathrm{B}(k,t)$ for $0<k\le1$ and $t\in\C$ are defined to be
\[
\mathrm{B}(1,t):=\sup_{\psi\in\Sigma}\beta_{\log\psi'}(t),\qquad
\mathrm{B}(k,t):=\sup_{\psi\in\Sigma^{\langle k\rangle}}\beta_{\log\psi'}(t).
\]
%Via holomorphic motion, any 
%$\psi\in\Sigma^{\langle k\rangle}$ is such that for a suitable constant $C_0$,
%the function $\psi+C_0$ may be fitted into a standard Beltrami solution family 
%$\Psi(\lambda,\cdot)$ at the parameter value $\lambda=k$. The correct value
%of the constant $C_0$ is $C_0:=\lim_{\zeta\to\infty}\zeta-\psi(\zeta)$.
%Moreover, for $\lambda\in\D$ close to $0$, the first term in the expansion
%\eqref{eq-holmot1} is dominant, and Theorem \ref{thm-main1} is seen to lead
%to a strong estimate of the integral means spectrum for the class
%$\Sigma^{\langle k\rangle}$, which we denote by $\mathrm{B}(k,t)$:
%\[
%\mathrm{B}(k,t):=\mathrm{B}_{\Sigma^{\langle k\rangle}}(t)=
%\sup_{\psi\in\Sigma^{\langle k\rangle}}\mathrm{B}_{\psi}(t).
%\]
As a consequence of Theorem \ref{thm-main}, we obtain an estimate of the 
universal spectrum $\mathrm{B}(k,t)$, which should be compared with the 
conjecture by Prause and Smirnov \cite{PS} that 
$B(k,t)=\frac{1}{4}k^2t^2$ for real $t$ with $|t|\le2/k$. Indeed, for small 
$k$, the estimate comes very close to the conjectured value.

%conjectured to be 
%asymptotically sharp as $k\to0^+$ by Prause and Smirnov \cite{PS}.

\begin{thm}
%For each $\epsilon$ with $0<\epsilon<1$, there exists a $k_0(\epsilon)$ with 
%$0<k_0(\epsilon)<1$, such that  
We have the following estimate:
\[
\mathrm{B}(k,t)\le \frac14\,k^2|t|^2(1+7k)^2,\qquad 
0<k<1,\,\,\,t\in\C.
\]
\label{thm-main2}
\end{thm}

The proof of this theorem is supplied in Subsection 
\ref{subsec-proofthmmain2}. In comparison, the estimate of Prause and Smirnov 
\cite{PS} applies only to real $t\ge t_k$, where $t_k:=2/(1+\sqrt{1-k^2})$. 

\begin{rem}
An argument based on the pointwise estimate 
\eqref{eq-Goluzinest2.6} combined with Theorem \ref{thm-main2} and H\"older's
inequality, like in the proof of Corollary \ref{cor-intmeast01}, shows 
that the estimate of Theorem \ref{thm-main2} can be improved for big values 
of $|t|$, for $t\in\C$:
\begin{equation}
\mathrm{B}(k,t)\le k|t|-\frac{1}{(1+7k)^2},\quad\text{for}\quad 
|t|\ge\frac{2}{k(1+7k)^2}.
\label{eq-off1.001}
\end{equation}
We may of course also combine this (Theorem \ref{thm-main2} and 
\eqref{eq-off1.001}) with the estimate of Prause and Smirnov \cite{PS}, 
using the convexity of the mapping $\C\ni t\mapsto\mathrm{B}(k,t)$ (convexity 
results from H\"older's inequality). The
result is a sharper estimate for complex $t$ near the real interval
$[t_k,+\infty[$, where $t_k=2/(1+\sqrt{1-k^2})$.
\end{rem}

%\begin{rem}
%This estimate is asymptotically as $k\to0$ congruent with the prediction of
%Prause and Smirnov \cite{PS}, which would have that $\mathrm{B}(k,t)\le
%\frac{1}{4}k^2|t|^2$, at least for real $t$, with equality for $|t|\le2/k$.
%For equality to have a chance to hold, one would need not to lose much in 
%the process which gives the estimate. This gives some indication of what 
%properties the approximal optimizing function $\varphi\in\classS$ ought to
%possess. In particular, in Marshall's estimate
%\eqref{eq-Marshall2}, the only way to have minimal loss for as $\alpha\to1^-$
%for a given function $\varphi\in\classS$ 
%is for 
%\[
%g_\varphi(r\zeta)\approx\frac{\bar t}{2\alpha}\log\frac{1}{1-r^2}
%\] 
%to hold where the mass of the density  
%\[
%\frac{1}{Z(r,\alpha)}\exp\Bigg\{\alpha\frac{|g_\varphi(r\zeta)|^2}
%{\log\frac{1}{1-r^2}}\Bigg\}
%\]
%is substantial on $\Te$. Here, $Z(r,\alpha)$ is a positive normalizing 
%constant so that we get a probability density.  
%\end{rem}

\subsection{Application to the Minkowski dimension of 
quasicircles}
If in the setting of the preceding subsection, we have a conformal mapping 
$\psi\in\Sigma^{\langle k\rangle}$ for some $0<k<1$, which means that $\psi$
has a $k$-quasiconformal extension that maps $\C\to\C$, it is of interest to
analyze the fractal dimension of the boundary 
$\Gamma_\psi:=\psi(\Te)$ in terms of $k$.
The fractal dimension of a curve $\Gamma$ can be measured by (i) the
upper Minkowski (or box-counting) dimension $\mathrm{dim}_M^+(\Gamma)$, 
(ii) the lower Minkowski (or box-counting) dimension 
$\mathrm{dim}_M^-(\Gamma)$, and (iii) the Hausdorff dimension 
$\mathrm{dim}_H(\Gamma)$. It is well-known that these dimensions are related:
\[
\mathrm{dim}_H(\Gamma)\le\mathrm{dim}_M^-(\Gamma)\le\mathrm{dim}_M^+(\Gamma),
\] 
where each inequality may be strict. 
Let us go to the level of universal dimension bounds:
\[
D_{M,1s}^+(k):=\sup_{\psi\in\Sigma^{\langle k\rangle}}\mathrm{dim}_M^+(\Gamma_\psi),\quad
D_{H,1s}(k):=\sup_{\psi\in\Sigma^{\langle k\rangle}}\mathrm{dim}_H(\Gamma_\psi);
\]
clearly, we have $D_{H,1s}(k)\le D_{M,1s}^+(k)$. Here, ``1s'' stands for 
one-sided, because $\psi$ is conformal inside the exterior disk $\D_e$ and 
$k$-quasiconformal off it. A symmetrization procedure which goes back to
Reiner K\"uhnau \cite{Kuh} (used by Stanislav Smirnov in \cite{Sm})
permits us to remove the one-sidedness, and to identify the 
curves $\Gamma_\psi$ where $\psi\in\Sigma^{\langle k\rangle}$ as $k'$-quasicircles
(a $k'$-quasicircle is the image of a circle under a $k'$-quasiconformal map)
where 
\begin{equation}
k=\frac{2k'}{1+(k')^2}.
\label{eq-dimequal0}
\end{equation}
This allows us to say that
\begin{equation}
D_{M,1s}^+(k)=D_{M}^+(k'),\quad D_{H,1s}(k)=D_{H}(k'),
\label{eq-dimequal1}
\end{equation}
where the right-hand expressions are the optimal universal dimension bounds
without one-sidedness. A result of Kari Astala \cite{Ast} says that these 
dimension bounds are all the same:  
\begin{equation*}
D_{M,1s}^+(k)=D_{M}^+(k')=D_{H,1s}(k)=D_{H}(k').
\label{eq-dimequal2}
\end{equation*}

\begin{cor}
We have that $D_H(k')=D_{M}^+(k')\le 1+(k')^2+\Ordo((k')^3)$ as $k\to0^+$.
\label{cor-dim}
\end{cor}

This is weaker than Smir\-nov's \cite{Sm} bound on the Hausdorff dimension: 
$D_{H}(k')\le1+(k')^2$. 
For completeness, the proof is supplied in Subsection \ref{subsec-dimbound}.
At the moment of writing this note, it has been announced by Oleg Ivrii
that for small $k$ and moderate $|t|$, the estimate of Theorem 
\ref{thm-main2} can be strengthened, and that as a consequence, $D_H(k')\le
1+(1-\epsilon_0)(k')^2$ for small $k'$, where $\epsilon_0$ is a small positive
constant. The constant $\epsilon_0$ appears from the estimation of the 
asymptotic variance of the unit ball of $\Pop L^\infty(\D)$ in \cite{Hed2}.

%\subsection{}

\subsection{Acknowledgements} First, I would like to thank Donald
Marshall for sharing his ideas with me on the topic surrounding Brennan's 
conjecture. Then I would like to also thank Kari Astala, Anton Baranov, 
Oleg Ivrii, Antti Per\"al\"a, Istv\'an Prause, Eero Saksman, and Serguei 
Shimorin for several valuable conversations on topics related with the 
present paper. 

\section{The Bloch space and duality}
%and probabilistic modelling}

\subsection{The Bloch space and the Bloch seminorm}
\label{subsec-BlochBloch}
The \emph{Bloch space}, which is named after Andr\'e Bloch 
\cite{Blochbiog}, consists of the holomorphic functions $g:\D\to\C$ subject
to the seminorm boundedness condition
\begin{equation}
\|g\|_{\mathcal{B}(\D)}:=\sup_{z\in\D}(1-|z|^2)|g'(z)|<+\infty.
\label{eq-Blochnorm}
\end{equation}
If, for $\zeta\in\D$,  $\phi_\zeta$ denotes the involutive M\"obius 
automorphism of $\D$ given by
\[
\phi_\zeta(z):=\frac{\zeta-z}{1-\bar\zeta z},
\]
then 
\[
\|g\circ\phi_\zeta\|_{\mathcal{B}(\D)}=\|g\|_{\mathcal{B}(\D)},\qquad \zeta\in\D,
\]  
which is easily obtained from the equality
\[
\frac{1-|\phi_\zeta(z)|^2}{|\phi'_\zeta(z)|}=1-|z|^2.
\]
Together with the rotations, these M\"obius involutions $\phi_\zeta$ generate 
the full automorphism group, which makes the Bloch seminorm invariant under 
all M\"obius automorphisms of $\D$. The subspace 
\[
\mathcal{B}_0(\D):=\big\{g\in\mathcal{B}(\D):\,\,
\lim_{|z|\to1^-}(1-|z|^2)|g'(z)|=0\big\} 
\]
is called the \emph{little Bloch space}. We shall be concerned here with the
extremal growth properties of Bloch functions, where functions in the little 
Bloch space are seen to grow too slowly. In other words, the properties will
take place in the quotient space $\mathcal{B}(\D)/\mathcal{B}_0(\D)$. 
An immediate observation we can make at this point is the following.

\begin{lem}
If $g\in\mathcal{B}(\D)$ with $g(0)=0$, then $g$ enjoys the growth estimate
\begin{equation*}
|g(z)|\le\|g\|_{\mathcal{B}(\D)}\int_0^{|z|}\frac{\diff t}{1-t^2}
=\frac12\,\|g\|_{\mathcal{B}(\D)}\log\frac{1+|z|}{1-|z|},\qquad z\in\D.
%\label{eq-pointwise1}
\end{equation*}
\label{lem-pointwise1}
\end{lem}

\begin{proof}
This is immediate from the integral formula
\[
g(z)=g(z)-g(0)=\int_0^{z}g'(\zeta)\diff\zeta,
\]
where the path is chosen to be the line segment connecting $0$ with $z$.
\end{proof}
%which is sharp pointwise.

\subsection{The Bloch space as the dual of the integrable 
quadratic differentials}
\label{subsec-PLinfty}

To a holomorphic quadratic differential $f(z)\diff z^2$ on the unit disk $\D$
we supply the norm
\[
\|f\|_{A^1(\D)}:=\int_\D|f(z)|\diff A(z),
\] 
and identify the holomorphic quadratic differentials with finite norm with 
the Bergman space $A^1(\D)$ (cf. \cite{Str}, p. 85). Here, slightly more 
generally, for $0<p<+\infty$,  we write $A^p(\D)$ for the Bergman space of all 
holomorphic functions $f:\D\to\C$ subject to the condition
\[
\|f\|_{A^p(\D)}:=\bigg(\int_\D|f(z)|^p\diff A(z)\bigg)^{1/p}<+\infty.
\] 
%Here, $p$ is assumed to be positive and real. 
Holomorphic quadratic differentials appear naturally in the context of 
Teichm\"uller theory.
If $\phi:\D\to\D$ is a M\"obius automorphism, while $f\in A^1(\D)$ and
$\mu\in L^\infty(\D)$ are given, then
\begin{equation}
\langle f,\mu\rangle_\D=\int_\D f\bar\mu\diff A=\int_\D
(f\circ\phi)\,(\bar\mu\circ\phi)\,|\phi'|^2\diff A
=\langle f_\phi,\mu_\phi\rangle_\D,
\end{equation}
where
\begin{equation}
f_\phi:=(\phi')^2f\circ\phi,\qquad \mu_\phi:=\frac{\phi'}{\bar\phi'}
\mu\circ\phi,
\end{equation}
so that while $f$ transforms as a quadratic differential, on the dual side
$\mu$ transforms as a $\diff z/\diff\bar z$-form. That is, $\mu$ reverses 
the complex structure. In fact, it acts to send a $(0,1)$-differential to a 
$(1,0)$-differential:
\[
h(z)\diff\bar z\,\, \mapsto\,\,\mu(z)h(z)\diff z.
\]  
For this reason, it will not come as a great surprise to us that such dual 
elements $\mu\in L^\infty(\D)$ are related with Beltrami equations and 
quasiconformal theory. 
Not all $\mu\in L^\infty(\D)$ give rise to nontrivial linear functionals on 
the space $A^1(\D)$. 
%Indeed, the null action space
%(or annihiliator)
%\begin{equation}
%N^\infty(\D):=\big\{\mu\in L^\infty(\D):\,\,
%\langle f,\mu\rangle_\D=0\,\,\,\text{for all}\,\,\,f\in A^1(\D)\big\},
%\end{equation}
%is infinite-dimensional. The space $N^\infty(\D)$ appears e.g. in the 
%context of quasiconformal deformation in (\cite{Ahl}, Lemma 1).
The nontrivial part of $\mu\in L^\infty(\D)$ may be represented by its Bergman 
projection $\Pop\mu$, given by \eqref{eq-Bergproj1.1}. On the
other hand, trivial such $\mu\in L^\infty(\D)$ appear in the context of 
quasiconformal deformation in, e.g., \cite{Ahl}, Lemma 1.
%From basic Functional Analysis, we know that the 
%dual space of $A^1(\D)$ with respect to the sesquilinear form $\langle\cdot,
%\cdot\rangle_\D$ may be identified isometrically with the quotient space 
%$L^\infty(\D)/N^\infty(\D)$. If 
%\[
%\Pop \mu(z):=\int_\D\frac{\mu(w)}{(1-z\bar w)^2}\,\diff A(w),\qquad z\in\D.
%\]
%denotes the \emph{Bergman projection}, which 
It is well-known that $\Pop$ acts boundedly on $L^p(\D)$ for each $p$ with
$1<p<+\infty$, and that $\Pop$ maps $L^\infty(\D)$ onto the Bloch space 
$\mathcal{B}(\D)$ (this result is from Coifman, Rochberg, and Weiss 
\cite{CRW}; see also, e.g., the book \cite{HKZ}).
%then $\Pop$ trivializes the null action space, that is, 
%$\Pop N^\infty(\D)=\{0\}$, and
%with slightly higher precision,
%\begin{equation*}
%N^\infty(\D)=\{\mu\in L^\infty(\D):\,\,
%\Pop\mu=0\}.
%\end{equation*}
%It is now apparent that $\Pop$ defines an isometric isomorphism
%\begin{equation*}
%\Pop:\,L^\infty(\D)/N^\infty(\D)\to\Pop L^\infty(\D),
%\end{equation*}
This suggests that we should equip the space $\mathcal{B}(\D)\cong 
\Pop L^\infty(\D)$ with the alternative norm 
%the canonical norm
\begin{equation}
\|g\|_{\Pop L^\infty(\D)}:=\inf\big\{\|\mu\|_{L^\infty(\D)}:\,\,
\mu\in L^\infty(\D)\,\,\,\text{ and }\,\,\, g=\Pop\mu\big\}.
\label{eq-Blochnorm1.2}
\end{equation}
When we do so, we write $\Pop L^\infty(\D)$ for the Bloch space. As such, 
$\Pop L^\infty(\D)$ is isometrically isomorphic with the dual space of 
$A^1(\D)$, with respect to the dual pairing $\langle\cdot,\cdot\rangle_\D$,
which needs to be understood in a generalized sense. The reason is that for 
$g\in\Pop L^\infty(\D)$ and $f\in A^1(\D)$, it might happen that $f\bar g\notin
L^1(\D)$, which would leave the dual action $\langle f,g\rangle_\D$ undefined.  
To remedy this defect, we consider the dilates $f_r(z):=f(rz)$ for $0<r<1$,
so that if $g=\Pop\mu$ where $\mu\in L^\infty(\D)$, we may \emph{define}
\[
\langle f,g\rangle_\D:=\lim_{r\to1^-}\langle f_r,g\rangle_\D,
\]
because then 
\begin{equation}
\langle f,g\rangle_\D=\langle f,\Pop\mu\rangle_\D=\langle f,\mu\rangle_\D,
\label{eq-duality1.001}
\end{equation}
as we see from the following calculation, which also justifies the 
existence of the limit:
\[
\langle f,g\rangle_\D:=\lim_{r\to1^-}\langle f_r,g\rangle_\D=
\lim_{r\to1^-}\langle f_r,\Pop\mu\rangle_\D=\lim_{r\to1^-}\langle\Pop f_r,
\mu\rangle_\D=\lim_{r\to1^-}\langle f_r,\mu\rangle_\D=\langle f,\mu\rangle_\D.
\]
Here, we use that the Bergman projection $\Pop$ is self-adjoint 
on $L^2(\D)$ and preserves $A^2(\D)$, and that we have the norm convergence 
$f_r\to f$ as $r\to1^-$ in the space $A^1(\D)$. 
In conclusion, \emph{we have identified the dual space of $A^1(\D)$ with the
%quotient space $L^\infty(\D)/N^\infty(\D)$ as well as with 
space $\Pop L^\infty(\D)$, isometrically and isomorphically, where the dual 
action is given by the sesquilinear form $\langle\cdot,\cdot\rangle_\D$}. 

%The fact that as a linear space, $\Pop L^\infty(\D)$ equals the Bloch space 
%$\mathcal{B}(\D)$, is from the work of Coifman, Rochberg, and Weiss 
%\cite{CRW}. However, the endowed norm is not the same as the standard seminorm 
%\eqref{eq-Blochnorm}. 
Recently, Antti Per\"al\"a \cite{Per} obtained the following estimate.

\begin{lem}
{\rm(Per\"al\"a)}
We have the inequality
\begin{equation*}
\|\Pop\mu\|_{\mathcal{B}(\D)}\le\frac{8}{\pi}\|\mu\|_{L^\infty(\D)},\qquad 
\mu\in L^\infty(\D),
%\label{eq-Perala}
\end{equation*}
where the constant $8/\pi$ is best possible. 
\label{lem-Perala}
\end{lem}

In the other direction, a less precise argument (see, e.g., Proposition 
\ref{prop-2.1.3} below) shows that 
a function $g\in\mathcal{B}(\D)$ with $\|g\|_{\mathcal{B}(\D)}\le1$ can be 
written in the form $g=\Pop\nu_g+G$, where $\nu_g\in L^\infty(\D)$ and 
$G\in H^\infty(\D)$, with the (semi)norm bounds $\|\nu_g\|_{L^\infty(\D)}\le1$ and
$\|G\|_{H^\infty(\D)}\le |g(0)|+6$.  
As the sharpness of
Per\"al\"a's estimate also comes from boundary effects, it would appear that 
modulo bounded terms, the unit ball of $\mathcal{B}(\D)$ 
can be mapped into the unit ball of $\Pop L^\infty(\D)$, whereas the unit ball
of $\Pop L^\infty(\D)$ is mapped into $\frac8\pi$ times the unit ball of 
$\mathcal{B}(\D)$. 
%A statement with a little more precision is offered in
%Lemma \ref{thm-main2}. 

%
%so essentially there would appear to be a gap of the size $8/\pi$ between
%the two norms.
%
% under the 
%\emph{Bergman projection} 
%\[
%\Pop f(z):=\int_\D\frac{f(w)}{(1-z\bar w)^2}\,\diff A(w),\qquad z\in\D,
%\]
%which is well-defined if $f\in L^1(\D)$. This perspective is of importance
%in Teichm\"uller theory. 
%
%There is another realization of the Bloch space

As for pointwise bounds, the analogue of Lemma \ref{lem-pointwise1} for
$\Pop L^\infty(\D)$ runs as follows.

\begin{lem}
Suppose that $\mu\in L^\infty(\D)$. Then 
\[
|\Pop\mu(z)|\le\|\mu\|_{L^\infty(\D)}\,\frac{1}{|z|^2}\log\frac{1}{1-|z|^2},
\qquad z\in\D.
\]
\label{lem-ptwise2}
\end{lem}

\begin{proof}
This follows from the estimate
\[
|\Pop\mu(z)|\le\int_\D\frac{|\mu(w)|}{|1-z\bar w|^2}\diff A(w)\le
\|\mu\|_{L^\infty(\D)}\int_\D\frac{1}{|1-z\bar w|^2}\diff A(w),\qquad z\in\D,
\]
by evaluation of the right-hand side integral.
\end{proof}

\begin{rem}
Both estimates of Lemmata \ref{lem-pointwise1} and \ref{lem-ptwise2} are
optimal. Moreover, since
\[
\lim_{r\to1^-}\frac{\frac{1}{r^2}\log\frac{1}{1-r^2}}{\frac12\log\frac{1+r}{1-r}}
=2,
\]
the permitted boundary growth is about twice as big for an element of the
unit ball of $\Pop L^\infty(\D)$ as for the unit ball of $\mathcal{B}(\D)$.
\end{rem}

%\end{document}
\section{Two notions of asymptotic variance and Marshall's 
estimate}
\label{sec-notionsvar}

%\end{document}
\subsection{Gaussian modelling and notions of asymptotic variance}
\label{subsec-Gaussmod} 
The standard normal rotationally invariant complex Gaussian distribution
$N_\C(0,1)$ has the probability measure $\e^{-|z|^2}\diff A(z)$ in the plane 
$\C$.  
More generally, we write $X\sim N_\C(m,\sigma^2)$ with mean  $m\in\C$ and 
rotationally invariant standard deviation $\sigma>0$ if 
\[
\frac{X-m}{\sigma}\sim N_\C(0,1). 
\] 
If $X\sim N_\C(0,\sigma^2)$, we may recover \emph{the variance} 
$\mathrm{var}\,X=\sigma^2$
from the formula
\begin{equation}
\mathrm{var}\,X:=\expect|X|^2,
\end{equation}
where $\expect$ stands for the expectation operation.
%of the random variable in question.
But we may also recover the variance from the tail distribution behavior as 
follows: $\sigma^2=\mathrm{tvar}\,X$, where $\mathrm{tvar}\,X$ is the 
\emph{tail variance}
\begin{equation}
\mathrm{tvar}\,X=\sigma^2=\inf\big\{\tau\in\R_+:
\,\,\expect\e^{|X|^2/\tau}<+\infty\big\}.
\end{equation}
Indeed, we see by direct inspection that
\begin{equation}
\expect\e^{|X|^2/\tau}=\sigma^{-2}\int_\C\e^{|z|^2/\tau}\e^{-|z|^2/\sigma^2}\diff A(z)=
\frac{\tau}{\tau-\sigma^2},\qquad \sigma^2<\tau<+\infty,
\label{eq-unifcontr1}
\end{equation}
which explodes as $\tau\to\sigma^2$. We might remark at this juncture that 
tail aspects of Gaussian densities are critical for the uncertainty principles
for Fourier transform pairs considered by Hardy and Beurling (see 
\cite{Har}, \cite{Horm}, \cite{Hed1}). 

%\end{document}

Makarov (see \cite{Mak1}, \cite{Mak2}, \cite{Mak3}) had the remarkable insight 
to model the boundary behavior of Bloch functions by Gaussian processes
(for a more directly probabilistic perspective, see Lyons' paper \cite{Lyo}).
For a typical Bloch function $g\in\mathcal{B}(\D)$ with $g(0)=0$ and given 
an $r$ with $0<r<1$, he thought of the dilates $g_r(\zeta)=g(r\zeta)$, 
for $\zeta\in\Te$, as an approximately rotationally invariant Gaussian 
stochastic variable, which in its turn evolves stochastically in time, where 
we think of time as related to the dilation parameter $r$ via 
\[
t=\log\frac{1}{1-r^2},\qquad \diff t=\frac{2r\diff r}{1-r^2}.
\]
So, taking this into account, we normalize the dilate, and let $X_r=X_r[g]$
be the function
\[
X_r(\zeta):=\frac{g(r\zeta)}{\sqrt{\log\frac{1}{1-r^2}}},\qquad \zeta\in\Te,
\,\,\,0<r<1,
\]
and since 
\[
\expect X_r=\int_\Te X_r(\zeta)\diff s(\zeta)=
\frac{g(0)}{\sqrt{\log\frac{1}{1-r^2}}}=0,
\]
we may calculate the variance from the formula 
\[
\mathrm{var}\, X_r[g]=\expect|X_r|^2=
\frac{\int_\Te|g(r\zeta)|^2\diff s(\zeta)}{\log\frac{1}{1-r^2}}. 
\]
The tail variance has no direct analogue, as the function $g_r$ is bounded
for fixed $r$ (see Lemma \eqref{lem-pointwise1}). However, in view of 
the uniform control observed in \eqref{eq-unifcontr1}, we can make sense of it 
asymptotically as $r\to1^-$. Following Makarov and Curtis McMullen \cite{McM},
we say that the Bloch function $g$ has the \emph{asymptotic variance}
\begin{equation}
\mathrm{avar}\,g:=\limsup_{r\to1^-}\expect|X_r|^2=
\limsup_{r\to1^-}
\frac{\int_\Te|g(r\zeta)|^2\diff s(\zeta)}{\log\frac{1}{1-r^2}},
\label{eq-asvar1}
\end{equation}
and the \emph{asymptotic tail variance}
\begin{equation}
\mathrm{atvar}\,g:=
\inf\big\{\tau\in\R_+:
\,\,\limsup_{r\to1^-}\expect\,\e^{|X_r|^2/\tau}<+\infty\big\},
\label{eq-astvar1}
\end{equation}
where the indicated expectation is given by
\begin{equation}
\expect\,\e^{|X_r|^2/\tau}:=\int_\Te\exp\Bigg(\frac{|g(r\zeta)|^2}
{\tau\log\frac{1}{1-r^2}}\Bigg)\diff s(\zeta).
\label{eq-astvar2}
\end{equation}
These asymptotic formulae apply also in the case when $g(0)\ne0$. 
%With a slight abuse of language, we shall say that the Bloch function $g$
%has the asymptotic variance given by \eqref{eq-asvar1}, and the tail variance
%given by \eqref{eq-astvar1}. 

%\end{document}

It is of interest to extend these notions of asymptotic variances to 
the setting of subsets $\mathcal{G}\subset\mathcal{B}(\D)$. To this end, 
we let the \emph{asymptotic (tail) variances} of $\mathcal{G}$ be
the supremum of the individual asymptotic (tail) variances:
\begin{equation}
\mathrm{avar}\,\mathcal{G}:=\sup_{g\in\mathcal{G}}
\,\mathrm{avar}\,g,\qquad \mathrm{atvar}\,\mathcal{G}=\sup_{g\in\mathcal{G}}
\,\mathrm{atvar}\,g,
%\frac{\int_\Te|g(r\zeta)|^2\diff s(\zeta)}{\log\frac{1}{1-r^2}},
\label{eq-asyvar1.1}
\end{equation}
but we also need uniform versions. We let the
\emph{uniform asymptotic variance} of $\mathcal{G}$ be the limit
\begin{equation}
\mathrm{avar}_u\,\mathcal{G}=\limsup_{r\to1^-}\sup_{g\in\mathcal{G}}\expect|X_r[g]|^2=
\limsup_{r\to1^-}\sup_{g\in\mathcal{G}}
\frac{\int_\Te|g(r\zeta)|^2\diff s(\zeta)}{\log\frac{1}{1-r^2}},
\label{eq-asvar1.1}
\end{equation}
and, analogously, the \emph{uniform asymptotic tail variance} of $\mathcal{G}$ 
is defined to be 
\begin{equation}
\mathrm{atvar}_u\,\mathcal{G}:=
\inf\Bigg\{\tau\in\R_+:
\,\,\limsup_{r\to1^-}\sup_{g\in\mathcal{G}}\int_\Te\exp\Bigg(\frac{|g(r\zeta)|^2}
{\tau\log\frac{1}{1-r^2}}\Bigg)\diff s(\zeta)<+\infty\Bigg\}.
\label{eq-astvar1.1}
\end{equation}
The way things are set up, we automatically have the inequalities
\[
\mathrm{avar}\,\mathcal{G}\le\mathrm{avar}_u\,\mathcal{G},\qquad
\mathrm{atvar}\,\mathcal{G}\le\mathrm{atvar}_u\,\mathcal{G},
\]
but we should expect that quite often, as a result of compactness, the 
above inequalities are equalities.
Note that for $F\in H^\infty(\D)$, we have that 
\[
\mathrm{avar}_u\,F\mathcal{G}\le\|F\|_{H^\infty(\D)}^2\mathrm{avar}_u\,
\mathcal{G},\qquad
\mathrm{atvar}_u\,F\mathcal{G}\le\|F\|_{H^\infty(\D)}^2\mathrm{atvar}_u\,
\mathcal{G},
\]
where the uniformity may be removed (by considering the functions 
individually).

%where the indicated expectation is given by
%\begin{equation}
%\expect\,\e^{|X_r|^2/\tau}:=\int_\Te\exp\Bigg\{\frac{|g(r\zeta)|^2}
%{\tau\log\frac{1}{1-r^2}}\Bigg\}\diff s(\zeta).
%\label{eq-astvar2}
%\end{equation}

%\end{document}
\subsection{Metric properties of the notions of asymptotic 
variance}
%We should analyze the stability of the asymptotic variance and the asymptotic
%tail variance under a bounded perturbation.
Let us consider the expressions, for $g\in\mathcal{B}(\D)$, 
\begin{equation}
\|g\|_{\mathrm{av}}:=(\mathrm{avar}\,g)^{1/2},\quad 
\|g\|_{\mathrm{atv}}:=(\mathrm{atvar}\,g)^{1/2}.
\label{eq-seminorm1}
\end{equation}

\begin{prop}
The functionals $\|\cdot\|_{\mathrm{av}}$ and $\|\cdot\|_{\mathrm{atv}}$ given by
\eqref{eq-seminorm1} are seminorms on $\mathcal{B}(\D)$. 
% 
%Suppose that $g_1,g_2\in\mathcal{B}(\D)$ and that $g_1-g_2\in H^\infty(\D)$. 
%Then $\mathrm{avar}\,g_1=\mathrm{avar}\,g_2$ and 
%$\mathrm{atvar}\,g_1=\mathrm{atvar}\,g_2$. 
\label{prop-stability1}
\end{prop}

\begin{rem}
In particular, for two functions $g_1,g_2\in\mathcal{B}(\D)$, we have the 
following:
(i) if $\mathrm{avar}(g_1-g_2)=0$, then 
$\mathrm{avar}\,g_1=\mathrm{avar}\,g_2$, and (ii) if 
$\mathrm{atvar}(g_1-g_2)=0$, 
then $\mathrm{atvar}\,g_1=\mathrm{atvar}\,g_2$.  
Examples of functions $g\in\mathcal{B}(\D)$ with
$\mathrm{avar}\,g=\mathrm{atvar}\,g=0$ include elements of $H^\infty(\D)$ 
as well as elements of the little Bloch space $\mathcal{B}_0(\D)$.  
\end{rem}

\begin{proof}[Proof of Proposition \ref{prop-stability1}]
The homogeneity property of the norm follows from the corresponding
property of the two asymptotic variances, which are easily verified by 
inspection:
\[
\mathrm{avar}\,\lambda g=|\lambda|^2\mathrm{avar}\,g,\quad
\mathrm{atvar}\,\lambda g=|\lambda|^2\mathrm{atvar}\,g,
\]
for  $g\in\mathcal{B}(\D)$. To obtain the remaining property (subadditivity),
that is,
\begin{equation}
\|g+h\|_{\mathrm{av}}\le\|g\|_{\mathrm{av}}+\|h\|_{\mathrm{av}},\quad
\|g+h\|_{\mathrm{atv}}\le\|g\|_{\mathrm{atv}}+\|h\|_{\mathrm{atv}},
\label{eq-subadd1.005}
\end{equation}
we begin with an elementary estimate.
For complex numbers $\xi,\eta\in\C$, and a positive real $\alpha$, we
have the estimate
\begin{equation}
|\xi+\eta|^2\le (1+\alpha)|\xi|^2+\bigg(1+\frac{1}{\alpha}\bigg)|\eta|^2.
\label{eq-trivial1}
\end{equation}
%where $\xi,\eta\in\C$, and $0<\alpha<+\infty$.
It follows from \eqref{eq-trivial1} that
%We write $h:=g_2-g_1\in H^\infty(\D)$, and observe that by \eqref{eq-trivial1},
%it follows that 
%\[
%|g_2(r\zeta)|^2\le (1+\epsilon)|g_1(r\zeta)|^2+
%\bigg(1+\frac{1}{\epsilon}\bigg)|h(r\zeta)|^2\le (1+\epsilon)|g_1(r\zeta)|^2+
%\bigg(1+\frac{1}{\epsilon}\bigg)\|h\|_{H^\infty(\D)}^2,
%\]
%and hence
\begin{equation}
\frac{|(g+h)(r\zeta)|^2}{\log\frac{1}{1-r^2}}
\le (1+\alpha)\frac{|g(r\zeta)|^2}{\log\frac{1}{1-r^2}}+
\bigg(1+\frac{1}{\alpha}\bigg)
\frac{|h(r\zeta)|^2}{\log\frac{1}{1-r^2}}.
\label{eq-relation1.01}
\end{equation}
If $\sigma_1,\sigma_2$ are positive reals such that 
$\mathrm{avar}\,g<\sigma_1^2$ and $\mathrm{avar}\,h<\sigma_2^2$, we integrate
with respect to $\diff s(\zeta)$ along $\Te$ in \eqref{eq-relation1.01}, 
and take the limsup as $r\to1^-$. The result is that
\[
\mathrm{avar}(g+h)\le
(1+\alpha)\sigma_1^2+\bigg(1+\frac{1}{\alpha}\bigg)
\sigma_2^2=(\sigma_1+\sigma_2)^2,
\] 
where in the last step, we made the optimal choice 
$\alpha=\sigma_2/\sigma_1$. By minimizing over
$\sigma_1,\sigma_2$, the first subadditivity in \eqref{eq-subadd1.005}
follows. 

Next, we let $\sigma_3,\sigma_4$ be positive reals with 
$\mathrm{atvar}\,g<\sigma_3^2$ and $\mathrm{atvar}\,h<\sigma_4^2$, 
and observe that 
by \eqref{eq-relation1.01},
\begin{equation*}
\exp\Bigg\{\frac{|(g+h)(r\zeta)|^2}
{(\sigma_3+\sigma_4)^2\log\frac{1}{1-r^2}}\Bigg\}
\le \exp\Bigg\{\frac{(1+\alpha)|g(r\zeta)|^2}
{(\sigma_3+\sigma_4)^2\log\frac{1}{1-r^2}}\Bigg\}
\exp\Bigg\{\frac{(1+\frac1{\alpha})|h(r\zeta)|^2}
{(\sigma_3+\sigma_4)^2\log\frac{1}{1-r^2}}\Bigg\},
%\label{eq-relation1.0101}
\end{equation*}
so that by H\"older's inequality,
\begin{multline*}
\int_\Te\exp\Bigg\{\frac{|(g+h)(r\zeta)|^2}
{(\sigma_3+\sigma_4)^2\log\frac{1}{1-r^2}}\Bigg\}\diff s(\zeta)
\le\Bigg(\int_\Te\exp\Bigg\{\frac{(1+\alpha)p|g(r\zeta)|^2}
{(\sigma_3+\sigma_4)^2\log\frac{1}{1-r^2}}\Bigg\}\diff s(\zeta)\Bigg)^{1/p}
\\
\times\Bigg(\int_\Te\exp\Bigg\{\frac{(1+\frac1{\alpha})p'|h(r\zeta)|^2}
{(\sigma_3+\sigma_4)^2\log\frac{1}{1-r^2}}\diff s(\zeta)\Bigg\}\Bigg)^{1/p'},
%\\
%=,
%\label{eq-relation1.0101}
\end{multline*}
where $p,p'$ are dual exponents. The choice $\alpha=\sigma_4/\sigma_3$,
$p=1+\alpha$, and $p'=1+\frac{1}{\alpha}$ gives 
\[
\frac{(1+\alpha)p}
{(\sigma_3+\sigma_4)^2}=\frac{1}{\sigma_3^2},\quad 
\frac{(1+\frac{1}{\alpha})p'}
{(\sigma_3+\sigma_4)^2}=\frac{1}{\sigma_4^2},
\] 
and allows us to conclude that 
\[
\mathrm{atvar}(g+h)\le(\sigma_3+\sigma_4)^2.
\]
By minimizing over $\sigma_3,\sigma_4$, the second subadditivity in 
\eqref{eq-subadd1.005} follows as well. The proof is complete.
\end{proof}

\begin{rem}
Let $\mathcal{N}_{\mathrm{av}}:=\{g\in\mathcal{B}(\D):\,\mathrm{avar}\,g=0\}$ and 
$\mathcal{N}_{\mathrm{atv}}:=\{g\in\mathcal{B}(\D):\,\mathrm{atvar}\,g=0\}$ be
the respective null subspaces for the seminorms \eqref{eq-seminorm1}. 
It would be natural to consider the Banach spaces which result from forming 
the completions of the quotient spaces 
$\mathcal{B}(\D)/\mathcal{N}_{\mathrm{av}}$ and 
$\mathcal{B}(\D)/\mathcal{N}_{\mathrm{atv}}$ with respect to the corresponding 
norms. 
\end{rem}

%\end{document}

\subsection{Makarov's growth estimate of a Bloch function}

The following result is immediate from the work of Makarov (see \cite{Mak1},
and \cite{Pombook}, Theorem 8.9 and Exercise 8.5.2); 
more or less the same argument can also be found in the work of Clunie 
and MacGregor \cite{CluMac}, but it is applied with less precision. 
It should be mentioned also that at
about the same time, Boris Korenblum \cite{K} found a cruder growth estimate 
than Makarov by studying dilatation as a map from the Bloch space to BMOA. 
In a sense, the present work may be viewed as a refinement of Korenblum's 
approach. 

\begin{thm} {\rm(Makarov)}
If $g\in\mathcal{B}(\D)$ with $g(0)=0$, then, for $0<r<1$, we have that
\[
\frac{\int_\Te|g(r\zeta)|^2\diff s(\zeta)}{\log\frac{1}{1-r^2}}
\le\|g\|^2_{\mathcal{B}(\D)}\quad\text{and}\quad 
\int_\Te\exp\Bigg\{\frac{|g(r\zeta)|^2}
{\tau\log\frac{1}{1-r^2}}\Bigg\}\diff s(\zeta)
\le\frac{\tau}{\tau-\|g\|_{\mathcal{B}(\D)}^2},
\]
provided that $\|g\|^2_{\mathcal{B}(\D)}<\tau<+\infty$. In particular, 
it follows that 
\[
\mathrm{avar}\,g\le\|g\|_{\mathcal{B}(\D)}^2\quad\text{and}\quad 
\mathrm{atvar}\,g\le\|g\|_{\mathcal{B}(\D)}^2.
\]
\label{thm-Makarov1}
\end{thm}

%\subsection{Comparison of the main result with Makarov's estimate}
%Note that Makarov's result (Theorem \ref{thm-Makarov1}) does not calculate 
%the asymptotic (tail) variannces $\mathrm{avar}\,g$ and $\mathrm{atvar}\,g$ 
%for a given $g\in\mathcal{B}(\D)$, it just provides a common upper bound 
%for them. However, 
%Makarov's estimate shows that
%$\mathrm{avar}_u\,\mathcal{G}_1\le1$ and $\mathrm{atvar}_u\,\mathcal{G}_1\le1$,
Let $\mathcal{G}_1$ and $\mathcal{G}_2$ denote the two unit balls 
\[
\mathcal{G}_1=\mathrm{ball}_0\,\mathcal{B}(\D):=\big\{g\in\mathcal{B}(\D):\,
\|g\|_{\mathcal{B}(\D)}\le1,\,\,\,g(0)=0\big\};
\]
and 
\[
\mathcal{G}_2=\mathrm{ball}\,\Pop L^\infty(\D):=\big\{g=\Pop\mu:\,
\|\mu\|_{L^\infty(\D)}\le1\big\}.
\]
In view of Proposition \ref{prop-stability1} as well as 
Proposition \ref{prop-2.1.3} below, we see that for all essential 
purposes, $\mathcal{G}_2$ is bigger than $\mathcal{G}_1$. Our main result, 
Theorem \ref{thm-main}, establishes that 
$\mathrm{atvar}_u\,\mathcal{G}_2=\mathrm{atvar}\,\mathcal{G}_2=1$. An 
inspiration for the present work is the paper \cite{AIPP} 
by Astala, Ivrii, Per\"al\"a, and Prause, where it was shown that 
$\mathrm{avar}\,\mathcal{G}_2\le1$; later, it was discovered that 
$\mathrm{avar}_u\,\mathcal{G}_2<1$ (see \cite{Hed2}). In particular, the
asymptotic variance and the asymptotic tail variance do not always coincide.
This may have some relation with dimension properties of quasicircles, 
see e.g. \cite{LZ}.
In comparison, Makarov's Theorem \ref{thm-Makarov1} obtains that 
$\mathrm{avar}_u\,\mathcal{G}_1\le1$ and $\mathrm{atvar}_u\,\mathcal{G}_1\le1$,
which together with Per\"al\"a's Lemma \ref{lem-Perala} only leads to the 
following rather weak estimates:  
$\mathrm{avar}_u\,\mathcal{G}_2\le64/\pi^2$ and 
$\mathrm{atvar}_u\,\mathcal{G}_2\le64/\pi^2=6.484\ldots$.
%, which is considerably
%worse  for the alternative unit ball
%\[
%\mathcal{G}_2=\mathrm{ball}\,\Pop L^\infty(\D):=\big\{g=\Pop\mu:\,
%\|\mu\|_{L^\infty(\D)}\le1\big\},
%\]
%which is for all essential purposes bigger than the standard Bloch ball
%$\mathcal{G}_1$ (see Proposition \ref{prop-2.1.3} below).
%In comparison, our main result (Theorem \ref{thm-main}) \emph{shows that}
%$\mathrm{atvar}_u\,\mathcal{G}_2=\mathrm{atvar}\,\mathcal{G}_2=1$, 
%\emph{whereas related work from \cite{Hed2} gives that} 
%$\mathrm{avar}_u\,\mathcal{G}_2<1$; in particular, the
%asymptotic variance and the asymptotic tail variance do not always coincide.

%\end{document}

\subsection{Marshall's estimate of the exponential type spectrum}
The following is the key observation of Marshall \cite{Mar1}. 

\begin{prop}
{\rm(Marshall)}
If $g\in\mathcal{B}(\D)$ has $\mathrm{atvar}\,g<\sigma^2$, where $\sigma$
is a positive real number, then
\begin{equation*}
\int_\Te|\e^{t g(r\zeta)}|\diff s(\zeta)
=\Ordo\big((1-r^2)^{-\sigma^2|t|^2/4}\big)\quad \text{as}\,\,\,\,\, r\to1^-,
%\int_\Te
%\exp\Bigg\{\alpha\frac{|g_\varphi(r\zeta)|^2}
%{\log\frac{1}{1-r^2}}\Bigg\}\diff s(\zeta).
\end{equation*}
%\label{eq-Marshall4}
%\label{eq-Marshall1.1}
%then
%\end{equation}
where the implied constant is uniform in $t\in\C$.
\label{prop-Marshall1}
\end{prop}

\begin{proof}
%So, we start with having
%\begin{equation}
%\limsup_{r\to1^-}\int_\Te\exp\Bigg\{\frac{|g(r\zeta)|^2}
%{\sigma_1^2\log\frac{1}{1-r^2}}\Bigg\}\diff s(\zeta)<+\infty 
%\label{eq-Marshall1.1}
%\end{equation}
%
Marshall \cite{Mar1} expands the following modulus-squared, for a
complex parameter $t\in\C$:
\begin{multline}
0\le\Bigg|\frac{g(r\zeta)}{\sigma\sqrt{\log\frac{1}{1-r^2}}}
-\frac{\sigma\bar t}{2}\sqrt{\log\frac{1}{1-r^2}}\Bigg|^2
\\
=\frac{|g(r\zeta)|^2}
{\sigma^2\log\frac{1}{1-r^2}}+\frac{\sigma^2|t|^2}{4}\log\frac{1}{1-r^2}
-\re (t g(r\zeta)).
\label{eq-Marshall2}
\end{multline}   
It is immediate from \eqref{eq-Marshall2} that
\begin{equation*}
\re (t g(r\zeta))\le\frac{|g(r\zeta)|^2}
{\sigma^2\log\frac{1}{1-r^2}}+\frac{\sigma^2|t|^2}{4}\log\frac{1}{1-r^2},
\label{eq-Marshall3}
\end{equation*}
and as we exponentiate both sides, and then integrate over the circle $\Te$, 
we arrive at
\begin{equation}
%\int_\Te \bigg|\bigg(\frac{r^2\zeta^2\varphi'(r\zeta)}{[\varphi(r\zeta)]^2}
%\bigg)^t\bigg|\diff s(\zeta)=
\int_\Te|\e^{t g(r\zeta)}|\diff s(\zeta)
%\\
\le(1-r^2)^{-\sigma^2|t|^2/4}\int_\Te
\exp\Bigg\{\frac{|g(r\zeta)|^2}
{\sigma^2\log\frac{1}{1-r^2}}\Bigg\}\diff s(\zeta),
\label{eq-Marshall4}
\end{equation}
from which the assertion of the proposition is immediate.
\end{proof}

This has the following consequence for the exponential type spectrum 
$\beta_g(t)$ of the function $\e^g$, where $g\in\mathcal{B}(\D)$.
 
\begin{cor} {\rm(Marshall)}
For $g\in\mathcal{B}(\D)$ and $0\le\sigma<+\infty$, we have the implication 
\[
\mathrm{atvar}\,g\le\sigma^2\,\,\,\,\Longrightarrow\,\,\,\,
\forall t\in\C:\,\,\beta_g(t)\le \frac{\sigma^2|t|^2}{4}.
\]
\label{cor-Marshall1.22}
\end{cor}

%\end{document}

\subsection{The estimate from above of the exponential 
type spectrum associated with a function in the unit 
ball of $\Pop L^\infty(\D)$}
\label{subsec-intmeans001}
For a function $g=\Pop\mu$, where $\mu\in L^\infty(\D)$, we need to estimate
from above the exponential type spectrum. The key estimate is the following.
To simplify the notation, we agree to write
\begin{equation}
I_g(a,r):=\int_\Te
\exp\Bigg\{a\frac{r^4|g(r\zeta)|^2}
{\log\frac{1}{1-r^2}}\Bigg\}\diff s(\zeta),\qquad 0<r<1.
\label{eq-Igint01}
\end{equation}

\begin{prop}
Let $g:=\Pop\mu$, where $\mu\in L^\infty(\D)$ with $\|\mu\|_{L^\infty(\D)}\le1$.
If $I_g(a,r)$ is the integral in \eqref{eq-Igint01}, we have the estimate
\[
\int_\Te|\e^{tr^2 g(r\zeta)}|\diff s(\zeta)\le
\begin{cases}
I_g(a,r)(1-r^2)^{-|t|^2/(4a)},\qquad |t|\le2a,
\\
I_g(a,r)(1-r^2)^{a-|t|},\qquad\quad\,\,\, |t|>2a.
\end{cases}
\]
\label{prop-betterest1.001}
\end{prop}

\begin{proof}
%To simplify the notation, let us agree to write
%\begin{equation*}
%I_g(a,r):=\int_\Te
%\exp\Bigg\{a\frac{r^4|g(r\zeta)|^2}
%{\log\frac{1}{1-r^2}}\Bigg\}\diff s(\zeta),\qquad 0<r<1.
%\end{equation*}
In view of \eqref{eq-Marshall4} and \eqref{eq-Igint01}, we have for positive 
$a$ that
\begin{equation}
\int_\Te|\e^{t_1r^2 g(r\zeta)}|\diff s(\zeta)
\le I_g(a,r)\,(1-r^2)^{-|t_1|^2/(4a)},\qquad 0<r<1,\,\,\,t_1\in\C.
\label{eq-betterest1.002}
\end{equation}
while, by the pointwise estimate of Lemma \ref{lem-ptwise2},
\begin{equation}
|\e^{t_2r^2 g(r\zeta)}|\le (1-r^2)^{-|t_2|},\qquad 0<r<1,\,\,\,t_2\in\C.
\label{eq-betterest1.003}
\end{equation}
For $t=t_1+t_2\in\C$, it follows from a combination of 
\eqref{eq-betterest1.002} and \eqref{eq-betterest1.003} that
\begin{equation}
\int_\Te|\e^{tr^2 g(r\zeta)}|\diff s(\zeta)=
\int_\Te|\e^{t_1r^2 g(r\zeta)}\e^{t_2r^2 g(r\zeta)}|\diff s(\zeta)
\le I_g(a,r)\,(1-r^2)^{-|t_2|-|t_1|^2/(4a)},\qquad 0<r<1,
\label{eq-betterest1.004}
\end{equation}
where we are free to optimize over all decompositions $t=t_1+t_2$.
For $|t|\le2a$, the best decomposition is $t=t_1+0$, that is, 
$t_1=t$ and $t_2=0$, while for $|t|>2a$, the best choice is $t_1=\theta t$
and $t_2=(1-\theta)t$, where $\theta:=2a/|t|$. 
After insertion into \eqref{eq-betterest1.004}, we arrive at the claimed
estimate.
\end{proof}

We may now supply the proof of Corollary \ref{cor-intmeast01} as an
application of our main theorem, Theorem \ref{thm-main}. 

\begin{proof}[Proof of Corollary \ref{cor-intmeast01}]
By Theorem \ref{thm-main}, we know that $I_g(a,r)\le C(a)$ for all $0<r<1$
and $0<a<1$. It follows from Proposition \ref{prop-betterest1.001} that 
for all $0<a<1$, 
\[
\int_\Te|\e^{tr^2 g(r\zeta)}|\diff s(\zeta)\le
\begin{cases}
C(a)(1-r^2)^{-|t|^2/(4a)},\qquad |t|\le2a,
\\
C(a)(1-r^2)^{a-|t|},\qquad\quad\,\,\, |t|>2a.
\end{cases}
\]
The factor $r^2$ in the exponent tends to $1$ as $r\to1^-$, and nothing 
changes drastically if it gets replaced by $1$. Finally, by letting $a$ 
approach $1$, we obtain the claimed estimate of the exponential type 
spectrum $\beta_g(t)$.
\end{proof}

%\end{document}
\section{Elementary properties of Bloch functions}

\subsection{The Bergman projection of an auxiliary bounded function}
For the proof of part (b) of Theorem \ref{thm-main}, we need to supply
the Bergman projection of a special function $\mu_0\in L^\infty(\D)$.

\begin{lem}
The Bergman projection of the function $\mu_0\in L^\infty(\D)$ given by
\[
\mu_0(z):=\frac{1-\bar z}{1-z},\qquad z\in\D,
\]
equals
\[
\Pop\mu_0(z)=\int_\D\frac{1-\bar w}{(1-w)(1-z\bar w)^2}\diff A(w)=
\frac{1}{z^2}\log\frac{1}{1-z}-\frac{1}{z},\qquad z\in\D\setminus\{0\}.
\]
\label{lem-calc1}
\end{lem}

The singularity at the origin is of course removable.

\begin{proof}[Proof of Lemma \ref{lem-calc1}]
This can be shown by direct computation of the Bergman integral. 
%An alternative path, however, is to consider the potential
%\[
%\Jop\mu_0(z):=z\int_\D\frac{1-\bar w}{(1-w)(1-z\bar w)}\diff A(w),
%\qquad z\in\C,
%\]
%which is easy to obtain explicitly by potential-theoretic considerations, 
%and to derive $\Pop\mu_0$ as the restriction to $\D$ of the 
%$\partial_z$-derivative of $\Jop\mu_0(z)$. 
\end{proof}

\begin{rem}
Along the segment $[0,1[\subset\D$, the function $\Pop\mu_0$
grows pretty much maximally quickly, compared with Lemma \ref{lem-ptwise2}:
\begin{multline*}
\Pop\mu_0(x)=
\frac{1}{x^2}\log\frac{1}{1-x}-\frac{1}{x}=
\frac{1}{x^2}\log\frac{1}{1-x^2}-\frac1{x^2}\big(x-\log(1+x)\big)
\\
\ge\frac{1}{x^2}\log\frac{1}{1-x^2}-\frac12,
\qquad 0<x<1.
\end{multline*}
\end{rem}

We will apply the following general estimate to the function $\mu=\mu_0$
of Lemma \ref{lem-calc1}.  

\begin{prop}
If $\mu\in L^\infty(\D)$ with $\|\mu\|_{L^\infty(\D)}\le1$, then, for $0<a<+\infty$,
we have
\begin{equation*}
(1-r^2)^{1/a}\int_\Te\big|\e^{r^2\zeta^2\Pop\mu(r\zeta)}\big|^2\diff s(\zeta)
\le\int_\Te\exp\Bigg\{a\frac{r^4|\Pop\mu(r\zeta)|^2}
{\log\frac{1}{1-r^2}}\Bigg\}\diff s(\zeta),\qquad 0<r<1.
%\label{eq-Marshall4'}
\end{equation*}
\label{prop-Marshall.est1}
\end{prop}

\begin{proof}
This is immediate from Marshall's inequality \eqref{eq-Marshall4}, with 
$t:=2$, $a:=1/\sigma^2$, and $g(z):=z^2\Pop\mu(z)$.
\end{proof}

It is now easy to obtain the sharpness part (b) of Theorem \ref{thm-main}.

\begin{cor}
If $\mu_0\in L^\infty(\D)$ is as in Lemma \ref{lem-calc1}, and $0<a<+\infty$, 
we have the estimate from below
\begin{equation*}
\int_\Te\exp\Bigg\{a\frac{r^4|\Pop\mu_0(r\zeta)|^2}
{\log\frac{1}{1-r^2}}\Bigg\}\diff s(\zeta)\ge\e^{-2}(1-r^2)^{-(a-1)/a}.
\end{equation*}
%so that, in particular, 
%\begin{equation*}
%\lim_{r\to1^-}\int_\Te\exp\Bigg\{a\frac{r^4|\Pop\mu_0(r\zeta)|^2}
%{\log\frac{1}{1-r^2}}\Bigg\}\diff s(\zeta)=1,\qquad 1<a<+\infty.
%\end{equation*}
\label{cor-main(b)}
\end{cor}

\begin{proof}
By Lemma \ref{lem-calc1}, 
\[
\e^{z^2\Pop\mu_0(z)}=\e^{-z}(1-z)^{-1},
\]
so that
\[
\int_\Te\big|\e^{r^2\zeta^2\Pop\mu_0(r\zeta)}\big|^2\diff s(\zeta)\ge\e^{-2}
\int_\Te|1-r\zeta|^{-2}\diff s(\zeta)=\e^{-2}(1-r^{2})^{-1}.
\]
The assertion of the corollary now follows rather immediately from Proposition 
\ref{prop-Marshall.est1}.
\end{proof}

\begin{rem}
Theorem \ref{thm-main}(b) now follows from the observation that
\[
\lim_{r\to1^-}\e^{-2}(1-r^2)^{-(a-1)/a}=+\infty,\qquad 1<a<+\infty.
\]
\end{rem}

%\end{document}
\subsection{The derivative of a Bloch functions}

%We supply the optimal pointwise growth estimate for the space 
%$\Pop L^\infty(\D)$.

%\begin{lem}
%Suppose that $\mu\in L^\infty(\D)$. Then 
%\[
%|\Pop\mu(z)|\le\|\mu\|_{L^\infty(\D)}\,\frac{1}{|z|^2}\log\frac{1}{1-|z|^2},
%\qquad z\in\D.
%\]
%\label{lem-ptwise1}
%\end{lem}
%
%\begin{proof}
%This follows from the estimate
%\[
%|\Pop\mu(z)|\le\int_\D\frac{|\mu(w)|}{|1-z\bar w|^2}\diff A(w)\le
%\|\mu\|_{L^\infty(\D)}\int_\D\frac{1}{|1-z\bar w|^2}\diff A(w),\qquad z\in\D,
%\]
%by evaluation of the right-hand side integral.
%\end{proof}
We first supply an elementary estimate which applies to the derivative 
of a Bloch function.

\begin{lem}
Suppose $f:\D\to\C$ is holomorphic, with $(1-|z|^2)|f(z)|\le 1$ on $\D$. 
%If in addition $f(0)=0$, 
Then we also have that
\[
(1-|z|^2)\omega(|z|)\bigg|\frac{f(z)-f(0)}{z}\bigg|\le1,
\]
where $\omega(t)=\frac13$ for $0\le t<\frac12$, and $\omega(t)=t/(2-t^2)$ 
for $\frac12\le t<1$.
\label{lem-2.1.2}
\end{lem}

\begin{proof}
%Since $f(0)=0$, 
The Cauchy integral formula applied to the dilate $f_r(\zeta)=f(r\zeta)$ 
gives that
\begin{multline*}
f(rz)-f(0)=f_r(z)-f_r(0)
=\int_\Te\bigg\{\frac1{1-z\bar w}-1\bigg\}f_r(w)\diff s(w)
\\
=\int_\Te\bigg\{\frac1{1-z\bar w}-1\bigg\}f_r(w)\diff s(w)
=z\int_\Te\frac{\bar w}{1-z\bar w}\,f_r(w)\diff s(w).
\end{multline*}
As a consequence, we obtain the estimate
\begin{multline*}
\bigg|\frac{f(rz)-f(0)}{rz}\bigg|\le
\frac{1}{r}\int_\Te\frac{1}{|1-z\bar w|}\,|f_r(w)|\diff s(w)
\le\frac{1}{r(1-r^2)}\int_\Te\frac{1}{|1-z\bar w|}\diff s(w)
\\
\le\frac{1}{r(1-r^2)}\,\frac{1}{|z|^2}\log\frac{1}{1-|z|^2},
\end{multline*}
for $z\in\D$ and $0<r<1$, which in its turn yields
\[
(1-|rz|^2)\bigg|\frac{f(rz)-f(0)}{rz}\bigg|\le
\frac{1-|rz|^2}{r(1-r^2)|z|^2}\log\frac{1}{1-|z|^2}.
\]
We plug in $r=2^{-1/2}$: 
\[
(1-|rz|^2)\bigg|\frac{f(rz)-f(0)}{rz}\bigg|\le
2^{1/2}\frac{2-|z|^2}{|z|^2}\log\frac{1}{1-|z|^2},\qquad r=2^{-1/2},
\]
and check that the right-hand side expression is an increasing function in 
the variable $|z|$. By restricting our attention to $|z|\le2^{-1/2}$, we find 
that
\[
(1-|\zeta|^2)\bigg|\frac{f(\zeta)-f(0)}{\zeta}\bigg|\le
2^{1/2}\frac{2-\frac12}{1/2}\log\frac{1}{1-\frac12}<3,\qquad |\zeta|\le\frac12.
\]
It is of course elementary that the following estimate holds:
\[
(1-|\zeta|^2)\bigg|\frac{f(\zeta)-f(0)}{\zeta}\bigg|\le
\frac{2-|\zeta|^2}{|\zeta|},\qquad 0<|\zeta|<1.
\]
The assertion of the lemma follows from a combination of these two estimates.
\end{proof}

%\end{document}

\subsection{Decomposition of a Bloch function}
It is well known that as sets, $\mathcal{B}(\D)=\Pop L^\infty(\D)$.
Here, we split a given Bloch function as an element of $\Pop L^\infty(\D)$
with small norm plus a smooth remainder. 

\begin{prop} Suppose $g\in\mathcal{B}(\D)$. Then there exists a 
$\nu_g\in L^\infty(\D)$ with $\|\nu_g\|_{L^\infty(\D)}\le\|g\|_{\mathcal{B}(\D)}$, such
that $g(z)=z^2\Pop\nu_g(z)+G(z)$, 
where $G\in H^\infty(\D)$ has $G'\in\mathcal{B}(\D)$, 
with the (semi)norm control
\[
\|G\|_{H^\infty(\D)}\le |g(0)|+6\|g\|_{\mathcal{B}(\D)},\quad
\|G'\|_{{\mathcal B}(\D)}\le 12\|g\|_{\mathcal{B}(\D)}.
\]
\label{prop-2.1.3}
\end{prop}

\begin{proof}
%We write 
%\[
%g_1(z):=g(z)-g(0)-g'(0)z,
%\]
%so that $g_1(0)=g_1'(0)=0$. 
We put
\[
\nu_g(z):=(1-|z|^2)\omega(|z|)\frac{g'(z)-g'(0)}{z},
\]
where $\omega(t)$ is as in Lemma \ref{lem-2.1.2}, and observe that by the
assertion of that lemma, $\|\nu_g\|_{L^\infty(\D)}\le\|g\|_{\mathcal{B}(\D)}$, as 
needed. If we write $\mu_g$ for the function
\[
\mu_g(z):=(1-|z|^2)\frac{g'(z)-g'(0)}{z},
\] 
then a direct calculation shows that
\begin{multline}
z^2\Pop\mu_g(z)
=z^2\int_\D\frac{(1-|w|^2)(g'(w)-g'(0))}{w(1-z\bar w)^2}\diff A(w)
\\
=\sum_{j=0}^{+\infty}(j+1)z^{j+2}\int_\D(1-|w|^2)\bar w^{j}\frac{g'(w)-g'(0)}{w}
\diff A(w)
\\
=\sum_{j=0}^{+\infty}(j+1)z^{j+2}\,\frac{\hat g(j+2)}{j+1}
=g(z)-g(0)-g'(0)z,\qquad z\in\D,
\label{eq-direct1.01}
\end{multline}
where the $\hat g(j)$ denote the Taylor coefficients of $g$. The difference 
$\mu_g-\nu_g$ may be written in the form
\[
\mu_g(z)-\nu_g(z)=(1-|z|^2)(1-\omega(|z|))\frac{g'(z)-g'(0)}{z}
=\frac{1-\omega(|z|)}{\omega(|z|)}\nu_g(z),
\]
which immediately yields the estimate
\[
|\mu_g(z)-\nu_g(z)|\le\frac{1-\omega(|z|)}{\omega(|z|)}\|\nu_g\|_{L^\infty(\D)}
\le\frac{1-\omega(|z|)}{\omega(|z|)}\|g\|_{\mathcal{B}(\D)}.
\]
A straightforward calculation tells us that
\[
\int_\D\frac{1-\omega(|w|)}{\omega(|w|)}\frac{\diff A(w)}{|1-z\bar w|^2}
\le5,\qquad z\in\D,
\]
and, that, as a consequence,
\[
\|\Pop(\mu_g-\nu_g)\|_{H^\infty(\D)}\le5\|g\|_{\mathcal{B}(\D)}.
\]
Finally, we split $g$ as $g=\Pop\nu_g+G$, where
\[
G(z):=g(0)+g'(0)z+\Pop(\mu_g-\nu_g)(z),
\]
and the claimed estimate $\|G\|_{H^\infty(\D)}\le |g(0)|+6\|g\|_{\mathcal{B}(\D)}$ 
follows. The estimate for the Bloch seminorm of $G'$ is obtained in a
similar manner.
\end{proof}

%\end{document}
\section{Identities for dilates of harmonic functions}
\label{sec-identities}

\subsection{An identity involving dilates of harmonic functions}
The following identity interchanges dilations, and although elementary, it 
is quite important.

\begin{lem}
Suppose $f,g:\D\to\C$ are two harmonic functions, which are are 
area-integrable: $f,g\in L^1(\D)$. Then we have that
\[
\int_\D f(rz)\bar g(z)\diff A(z)=\int_\D f(z)\bar g(rz)\diff A(z),\qquad
0<r<1.
\] 
\label{lem-basic1}
\end{lem}

\begin{proof}
Both integrals are well-defined, since $f,g\in L^1(\D)$ and the dilates
$f_r(z)=f(rz)$, $g_r(z)=g(rz)$, are bounded for $0<r<1$. If we consider also
the dilates $f_\varrho,g_\varrho$ for $0<\varrho<1$, we may use Fourier methods 
to establish the identities 
\begin{equation}
\int_\D f(r\varrho z)\bar g(\varrho z)\diff A(z)
=\sum_{j\in\Z}\frac{\varrho^{2|j|}r^{|j|}}{|j|+1}
\hat f(j)\overline{\hat g(j)},
\label{eq-basic1.1}
\end{equation}
and
\begin{equation}
\int_\D f(\varrho z)\bar g(r\varrho z)\diff A(z)
=\sum_{j\in\Z}\frac{\varrho^{2|j|}r^{|j|}}{|j|+1}\hat f(j)\overline{\hat g(j)},
\label{eq-basic1.2}
\end{equation}
so that
\[
\int_\D f(r\varrho z)\bar g(\varrho z)\diff A(z)=
\int_\D f(\varrho z)\bar g(r\varrho z)\diff A(z).
\] 
Here, we use $\hat f(j),\hat g(j)$ to denote the Fourier coefficients of the
functions $f,g$, considered as distributions on the circle $\Te$. The claimed
identity now follows by letting $\varrho\to1$, since $f_\varrho\to f$ and
$g_\varrho\to g$ in $L^1(\D)$, while $f_{r\varrho}\to f_r$ and
$g_{r\varrho}\to g_r$ in $L^\infty(\D)$. 
\end{proof}

\subsection{An identity involving dilates which connects the inner 
products on the circle and the disk} The following identity is 
key to our analysis. 

\begin{lem}
Suppose $g,h:\D\to\C$ are functions, where $g$ is holomorphic
and $h$ is harmonic. 
If $g\in L^1(\D)$ and $h$ is the Poisson integral of a function in $L^1(\Te)$,
then we have that
\[
\langle g_r,\bar z h\rangle_\Te=\langle g,(\partial h)_r\rangle_\D,
\]
where we write $f_r$ for the dilate of the function $f$: 
$f_r(\zeta)=f(r\zeta)$.
\label{lem-basic2}
\end{lem}

\begin{proof}
As in the proof of Lemma \ref{lem-basic1}, we let $\hat g(j),\hat h(j)$ 
denote the Fourier coefficients of the boundary distributions associated
with $g,h$ on $\Te$. By the Plancherel identity, then, we know that
\[
\langle g_r,\bar z h\rangle_\Te=\int_\Te \zeta g(r\zeta)\bar h(\zeta)
\diff s(\zeta)=\sum_{j=0}^{+\infty}\hat g(j)\overline{\hat h(j+1)}\,r^j.
\]
On the other hand, since
\[
(\partial h)(\zeta)=\sum_{j=0}^{+\infty}(j+1)\hat h(j+1)\zeta^j,\qquad \zeta\in\D,
\]
it follows from \eqref{eq-basic1.1} by letting $\varrho\to1^-$ that
\[
\langle g,(\partial h)_r\rangle_\D=\int_\D g(z)\overline{(\partial h)(rz)}
\diff A(z)=\sum_{j=0}^{+\infty}\frac{r^j}{j+1}
(j+1)\hat g(j)\overline{\hat h(j+1)}=
\sum_{j=0}^{+\infty}\hat g(j)\overline{\hat h(j+1)}\,r^j.
\]
The assertion of the lemma follows.
\end{proof}

%\end{document}

\section{Dilational reverse isoperimetry: Hardy and Bergman}

\subsection{The isoperimetric inequality of Carleman}
The classical isoperimetric inequality says that the area enclosed by 
a closed loop of length $L$ is at most $L^2/(4\pi)$. Torsten Carleman 
(see \cite{Car}, \cite{Str}) found a nice analytical approach to this fact
which gave the estimate
\begin{equation}
\|f\|_{A^{2p}(\D)}\le\|f\|_{H^p(\D)},\qquad f\in H^p(\D), 
\label{eq-isoper1}
\end{equation}
for $0<p<+\infty$. 
Here, $H^1(\D)$ is the $p=1$ instance of the classical \emph{Hardy space}
$H^p(\D)$, for $0<p\le+\infty$. For $0<p<+\infty$, $H^p(\D)$ consists all 
holomorphic functions $f:\D\to\C$ subject to the norm boundedness condition
\[
\|f\|_{H^p(\D)}^p:=\sup_{0<r<1}\int_\Te|f(r\zeta)|^p\diff s(\zeta)
<+\infty.
\]
It is well-known that for $f\in H^p(\D)$, the function has well-defined 
nontangential boundary values a.e., and that the norm is attained for $r=1$:
\[
\|f\|_{H^p(\D)}^p=\int_\Te|f(\zeta)|^p\diff s(\zeta).
\]

\subsection{A similar reverse inequality for dilates}

%\section{Estimation of a partial $L^1$-Dirichlet integral on a 
%concentric disk}
%
%\subsection{The setting}

The Carleman estimate \eqref{eq-isoper1} has a reverse if we take the Hardy 
norm of a dilate $f_r(\zeta):=f_r(\zeta)$ in place of the function. 
%\eqref{eq-isoper3.01} 
Here, we will not dwell on that matter, but instead look for a somewhat
similar reverse estimate. 
%can be extended by introducing a suitable 
%weight function. 
We consider a nontrivial harmonic function $\hfun:\D\to\C$, and 
%let the holomorphic derivative $\partial\hfun$ play 
%the role of the previous function $f$ of \eqref{eq-isoper3.01}, while a 
%suitable power of $|\hfun|$ is the weight. Then, as in 
%\eqref{eq-isoper3.01}, 
obtain from the Cauchy-Schwarz inequality that
\begin{multline}
\|(\partial \hfun)_r\|_{A^1(\D)}=\int_\D|(\partial \hfun)(r\zeta)|\diff A(\zeta)
=\frac{1}{r^2}\int_{\D(0,r)}|\partial \hfun(z)|\diff A(z)
\\
\le\frac{1}{r^2}\Bigg(\int_{\D(0,r)}\frac{|\hfun(z)|^{\theta}}{1-|z|^2}\diff A(z)
\Bigg)^{1/2}\Bigg(\int_{\D(0,r)}
\frac{|\partial \hfun(z)|^{2}}{|\hfun(z)|^{\theta}}
(1-|z|^2)\diff A(z)\Bigg)^{1/2},
%\\
%=\frac{1}{r^2}\bigg(\log\frac{1}{1-r^2}\bigg)^{1/2}
%\Bigg(\int_{\D(0,r)}\Bigg(\int_{\D(0,r)}|f(z)|^2(1-|z|^2)\diff A(z)\Bigg)^{1/2}
%\\
%\le\frac{1}{r^2}\bigg(\log\frac{1}{1-r^2}\bigg)^{1/2}
%\Bigg(\int_{\D}|f(z)|^2(1-|z|^2)\diff A(z)\Bigg)^{1/2}
%=\frac{1}{r^2}\bigg(\log\frac{1}{1-r^2}\bigg)^{1/2}
%\|f\|_{A^2_1(\D)}.
\label{eq-isoper3.02}
\end{multline}
%as a result of the Cauchy-Schwarz inequality.
where $\theta$ is a real parameter, which we shall  confine to interval 
$0\le\theta<1$. 
\section{Duality and the estimate of the uniform asymptotic 
tail variance}

\subsection{Green's formula}

We recall Green's formula for the unit disk:
\[
\int_\D (u\hDelta v-v\hDelta u)\diff A=\frac{1}{2}\int_\Te 
(u\partial_{\mathrm{n}}v-v\partial_{\mathrm{n}}u)\diff s, 
\]
where $u,v$ are both assumed $C^2$-smooth in the closed unit disk $\bar\D$,
and the normal derivative is in the exterior direction. The constant $\frac12$
is the result of our normalizations. If we choose $v(z):=1-|z|^2$, the formula
simplifies to
\begin{equation}
\int_\D u\diff A+\int_\D (1-|z|^2)\hDelta u(z)\diff A(z)=\int_\Te u\diff s. 
\label{eq-Green1}
\end{equation}
Next, we let $h:\bar\D\to\C$ be $C^2$-smooth and harmonic in $\D$, 
let $\epsilon$ denote a small positive constant, and consider the function
\[
u(z):=(|h(z)|^2+\epsilon)^{s},
\]
where it is assumed that $\frac12<s\le 1$. Then $u$ is $C^2$-smooth on $\bar\D$,
and we may calculate its Laplacian:
\begin{multline}
\hDelta u=s\bigg\{(|h|^2+\epsilon)^{s-1}\hDelta|h|^2
-(1-s)(|h|^2+\epsilon)^{s-2}|\partial|h|^2|^2\bigg\}
\\
=s(|h|^2+\epsilon)^{s-2}\bigg\{(|h|^2+\epsilon)(|\partial h|^2
+|\bar\partial h|^2)
-(1-s)|\bar h\partial h+h\partial \bar h|^2\bigg\}.
\label{eq-Laplace1.1}
\end{multline}
For complex numbers $a,b\in\C$ and a positive real number $t\in\R_+$, 
we know by direct algebraic manipulation
that
\[
|a+b|^2=|a|^2(1+t^2)+|b|^2(1+t^{-2})-|at-bt^{-1}|^2.
\]
As we apply this identity in the setting of \eqref{eq-Laplace1.1},
with $a:=\bar h\partial h$ and $b:=h\partial\bar h$, we find that
\begin{multline}
\hDelta u
=s(|h|^2+\epsilon)^{s-2}\bigg\{\epsilon(|\partial h|^2+|\bar\partial h|^2)
+[1-(1-s)(1+t^2)]|h\partial h|^2
\\
+[1-(1-s)(1+t^{-2})]|h\bar\partial h|^2
+(1-s)|t\bar h\partial h-t^{-1}h\partial \bar h|^2\bigg\}
\label{eq-Laplace1.2}
\end{multline}
We choose $t$ so that we may suppress the term with  $|h\bar\partial h|^2$,
which happens for $t:=\sqrt{(1-s)/s}$. It now follows from 
\eqref{eq-Laplace1.2} that
\begin{multline}
\hDelta u
=s(|h|^2+\epsilon)^{s-2}\bigg\{\epsilon(|\partial h|^2+|\bar\partial h|^2)
+\frac{2s-1}{s}\,|h\partial h|^2
+s\bigg|\frac{1-s}{s}\bar h\partial h-h\partial \bar h\bigg|^2\bigg\}
\\
\ge(2s-1)(|h|^2+\epsilon)^{s-2}|h\partial h|^2.
\label{eq-Laplace1.3}
\end{multline}
As we insert this inequality into the identity \eqref{eq-Green1}, the result 
is that
\begin{multline}
\int_\D(|h|^2+\epsilon)^s\diff A+
(2s-1)\int_\D (1-|z|^2)(|h(z)|^2+\epsilon)^{s-2}|h(z)\partial h(z)|^2
\diff A(z)
\\
\le \int_\D u\diff A+\int_\D(1-|z|^2)\Delta u(z)
=\int_\Te u\diff s=\int_\Te(|h|^2+\epsilon)^s\diff s.
\label{eq-Green2}
\end{multline}

We recall the Hardy space $h^q(\D)$ of harmonic functions. Since we will
only be interested in exponents in the range $1<q\le2$, these are just the
Poisson extensions of boundary functions in $L^q(\Te)$: $h^q(\D)\cong
L^q(\Te)$, isometrically and isomorphically.   

\begin{prop}
{$(1<q\le 2)$}
Suppose $h:\D\to\C$ is the Poisson extension to the disk $\D$ of a boundary 
function in $L^q(\Te)$, also denoted by $h$.  Then, unless $h$ vanishes 
identically, it enjoys the estimate
\begin{equation*}
\int_\D|h|^q\diff A+
(q-1)\int_\D (1-|z|^2)\frac{|\partial h(z)|^2}{|h(z)|^{2-q}}
\diff A(z)
\le\int_\Te|h|^q\diff s.
%\label{eq-Green3}
\end{equation*}
\label{prop-Green1}
\end{prop}

\begin{proof}
We apply the estimate \eqref{eq-Green2} to the dilates $h_r(z)=h(rz)$, for
$0<r<1$, and use $s=q/2$. Then as $r\to1^-$ and $\epsilon\to0^+$, 
\[
\int_\Te(|h_r|^2+\epsilon)^{q/2}\diff s\to\int_\Te|h|^q\diff s, 
\]
and Fatou's lemma tells gives us the necessary control of the left-hand side.
\end{proof}

%\begin{rem}
%The most interesting case is probably when $h$ is real-valued. For 
%real-valued $h$, the zero set $Z(h):=\{z\in\D:\,h(z)=0\}$ is seen to be
%a union of curves, and the above estimate supplies some amount of geometric 
%control on $Z(h)$. Indeed, it may be argued that 
%\[
%h(z)\asymp |\partial h(z)|\,\mathrm{dist}(z,Z(h))
%\]
%at least locally, which then can get implemented in the second integral on 
%the left-hand side.
%\end{rem}

%\end{document}
\subsection{Hardy space methods and the weighted 
Bergman reverse Carleman isoperimetric inequality} 
We now turn the estimate \eqref{eq-isoper3.02} into a tool for effective
control of a function $h$ in the harmonic Hardy space $h^q(\D)$ for
$1<q\le2$. First, we implement \eqref{eq-isoper3.02} with $\theta=2-q$: 
\begin{multline}
\|(\partial \hfun)_r\|_{A^1(\D)}
\le\frac{1}{r^2}\Bigg(\int_{\D(0,r)}\frac{|\hfun(z)|^{2-q}}{1-|z|^2}\diff A(z)
\Bigg)^{1/2}\Bigg(\int_{\D(0,r)}
\frac{|\partial \hfun(z)|^{2}}{|\hfun(z)|^{2-q}}
(1-|z|^2)\diff A(z)\Bigg)^{1/2},
\label{eq-isoper3.03}
\end{multline}
for $0<r<1$, where all we ask of $h:\D\to\C$ is that it is harmonic in $\D$.
Using polar coordinates, we see that
\begin{multline}
\int_{\D(0,r)}\frac{|\hfun(z)|^{2-q}}{1-|z|^2}\diff A(z)=\int_0^r\int_\Te
|h_\varrho(\zeta)|^{2-q}\diff s(\zeta)\frac{2\varrho\diff\varrho}{1-\varrho^2}
\\
\le\sup_{0<\varrho<r}\|h_\varrho\|_{L^{2-q}(\Te)}^{2-q}\int_0^r
\frac{2\varrho\diff\varrho}{1-\varrho^2}
=\sup_{0<\varrho<r}\|h_\varrho\|_{L^{2-q}(\Te)}^{2-q}\log\frac{1}{1-r^2}
\label{eq-isoper3.04}
\end{multline}
We recognize on the right-hand side the harmonic Hardy space $h^{2-q}(\D)$
quasinorm of the dilate $h_r$. From H\"older's inequality and the restriction
$1<q\le2$, we know that 
$\|h_\varrho\|_{L^{2-q}(\Te)}\le\|h_\varrho\|_{L^{1}(\Te)}$, and, in addition, 
the norms $\|h_\varrho\|_{L^{1}(\Te)}$ are known to increase with the radius 
$\varrho$. So, it follows from \eqref{eq-isoper3.04} that
\begin{equation}
\int_{\D(0,r)}\frac{|\hfun(z)|^{2-q}}{1-|z|^2}\diff A(z)\le
\|h_r\|_{L^{1}(\Te)}^{2-q}\log\frac{1}{1-r^2}.
\label{eq-isoper3.05}
\end{equation}
Next, we assume $h$ is the Poisson extension to the disk $\D$ of a function 
in $L^q(\Te)$, which we also denote by $h$. Then, by Proposition 
\ref{prop-Green1}, we know that
\begin{equation}
\int_\D (1-|z|^2)\frac{|\partial h(z)|^2}{|h(z)|^{2-q}}
\diff A(z)
\le\frac{1}{q-1}\bigg\{\int_\Te|h|^q\diff s-\int_\D|h|^q\diff A\bigg\},
\label{eq-isoper3.06}
\end{equation}
and by inserting the estimates \eqref{eq-isoper3.05} and \eqref{eq-isoper3.06}
into \eqref{eq-isoper3.03}, we obtain that
\begin{equation}
\|(\partial \hfun)_r\|_{A^1(\D)}
\le\frac{1}{r^2}\|h\|_{L^1(\Te)}^{1-\frac{q}{2}}
\bigg\{\frac{1}{q-1}\bigg(\int_\Te|h|^q\diff s-\int_\D|h|^q\diff A\bigg)
\bigg\}^{1/2}
\sqrt{\log\frac{1}{1-r^2}}.
\label{eq-isoper3.07}
\end{equation}

We will refer to minus the differential entropy as the 
\emph{differential anentropy}, and as the area-$L^1$ norm of the dilatation of 
the gradient gets controlled in terms of this quantity, we name the result 
accordingly.   

\begin{thm}
{\rm(differential anentropy bound)}
Suppose $h:\D\to\R$ is the Poisson extension to the disk of a function 
in $L^p(\Te)$, for some $p$ with $1<p\le2$. The boundary function is also 
denoted by $h$. If $h\ge0$ on $\D$, and if $h(0)=1$, 
then
\begin{equation*}
\|(\partial \hfun)_r\|_{A^1(\D)}
\le\frac{1}{r^2}
\bigg\{\int_\Te h\log h\,\diff s\bigg\}^{1/2}
\sqrt{\log\frac{1}{1-r^2}},\qquad 0<r<1.
%\label{eq-isoper3.08}
\end{equation*}
\label{thm-entropy1}
\end{thm}

\begin{proof}
The function
\[
F(t,q):=\frac{t^q-t}{q-1}, \qquad 0\le t<+\infty,\,\,\,1<q<2,
\]
is strictly decreasing as a function of $q$, with limit
\[
F(t,1):=\lim_{q\to1^+}F(t,q)=t\log t,
\] 
understood as $F(0,1)=0$ for $t=0$. Although $F(t,1)$ attains negative values
for $0<t<1$, it is easy to see that 
$F(t,q)\ge F(t,1)\ge -\e^{-1}$. Since it is given that $h\ge0$, we know that 
$|h|=h$, and by the subharmonicity of the function $h^q$, 
\[
\int_\D |h|^q\diff A=\int_\D h^q\diff A\ge h(0)^q=1,
\]
so that 
\begin{multline*}
\frac{1}{q-1}\bigg(\int_\Te|h|^q\diff s-\int_\D|h|^q\diff A\bigg)
\le\frac{1}{q-1}\bigg(\int_\Te h^q\diff s-1\bigg)
=\frac{1}{q-1}\int_\Te (h^q-h)\diff s=\int_\Te F(h,q)\diff s.
\end{multline*}
By the monotone convergence theorem applied to the positive functions 
$F(h,q)+\e^{-1}$, we see that
\[
\lim_{q\to1^+}\int_\Te F(h,q)\diff s=\int_\Te F(h,1)\diff s=\int_\Te h\log h\diff s
\] 
provided that $h\in L^p(\Te)$ for some $p$ with $1<p\le2$, so that the 
left-hand side limit is finite. By letting $q\to1^+$ in \eqref{eq-isoper3.07}, 
the claimed estimate follows.
\end{proof}

%\end{document}
\subsection{Applications of duality techniques to the 
dilates of Bloch functions} 
In view of Lemma \ref{lem-basic2}, combined with the equality 
\eqref{eq-duality1.001}, we have, for $\mu\in L^\infty(\D)$, 
$g:=\Pop\mu\in\mathcal{B}(\D)$, and a harmonic function $h$ on $\D$, which 
is the Poisson integral of an $L^1(\Te)$ function, also denoted by $h$, 
\[
\langle zg_r,h\rangle_\Te=\langle g,(\partial h)_r\rangle_\D
%=\langle g,(\partial h)_r\rangle_\D
=\langle\Pop\mu,(\partial h)_r\rangle_\D=\langle\mu,(\partial h)_r\rangle_\D,
\]
for $0<r<1$.
If, in addition, $h\ge0$ on $\D$, $h(0)=1$, and the boundary values are 
in $L^2(\Te)$, 
then Theorem \ref{thm-entropy1} gives that 
\begin{multline}
|\langle zg_r,h\rangle_\Te|=|\langle\mu,(\partial h)_r\rangle_\D|
\le\|\mu\|_{L^\infty(\D)}\|(\partial\hfun)_r\|_{A^1(\D)}
\\
\le\frac{\|\mu\|_{L^\infty(\D)}}{r^2}
\bigg\{\int_\Te h\log h\,\diff s\bigg\}^{1/2}
\sqrt{\log\frac{1}{1-r^2}},\qquad0<r<1.
\label{eq-isoper3.08}
\end{multline}

\begin{thm}
Suppose $g=\Pop\mu$, where $\mu\in L^\infty(\D)$, and consider, for
$0<r<1$ and $0\le\eta<+\infty$, the set
\[
\mathcal{E}(r,\eta):=\big\{\zeta\in\Te:\,\,\re(\zeta g(r\zeta))\ge\eta\big\}.
\]
Then $|\mathcal{E}(r,\eta)|_s$, the $s$-length of this set, enjoys the
bound
\[
|\mathcal{E}(r,\eta)|_s\le\exp\Bigg\{-\frac{r^4\eta^2}
{\|\mu\|_{L^\infty(\D)}^2\log\frac{1}{1-r^2}}\Bigg\}.
\]
\label{thm-weak1}
\end{thm}

\begin{proof}
We let $h$ be the Poisson extension of the boundary function which equals 
$1/|\mathcal{E}(r,\eta)|_s$ on $\mathcal{E}(r,\eta)$ and vanishes off 
$\mathcal{E}(r,\eta)$. Then $h\ge0$ on $\D$, and $h(0)=1$, and the boundary 
function is in $L^\infty(\Te)$, so we are in a position to apply 
\eqref{eq-isoper3.08}. As it turns out, the indicated estimate is a direct 
consequence of \eqref{eq-isoper3.08}. 
\end{proof}

%\end{document}

\begin{cor}
Suppose $g=\Pop\mu$, where $\mu\in L^\infty(\D)$, and consider, for
$0<r<1$ and $0\le\eta<+\infty$, the set
\[
\mathcal{E}_N(r,\eta):=\big\{\zeta\in\Te:\,\,\max_k\re[\omega^k\zeta g(r\zeta)]
\ge\eta\big\},
\]
where $\omega:=\e^{\imag 2\pi/N}\in\Te$ is a root of unity, for some integer 
$N=1,2,3,\ldots$.
Then the $s$-length of this set enjoys the
bound
\[
|\mathcal{E}_N(r,\eta)|_s\le N\exp\Bigg\{-\frac{r^4\eta^2}
{\|\mu\|_{L^\infty(\D)}^2\log\frac{1}{1-r^2}}\Bigg\}.
\]
\label{cor-weak1}
\end{cor}

\begin{proof}
The assertion is immediate from Theorem \ref{thm-weak1}, since the set
$\mathcal{E}_N(r,\eta)$ may be split as the union of $N$ sets, each of
which may be estimated using Theorem \ref{thm-weak1}.
\end{proof}

\begin{lem}
For $0<r<1$ and $0\le\eta<+\infty$, the set 
\[
\mathcal{F}(r,\eta):=\big\{\zeta\in\Te:\,\,|g(r\zeta)|\ge\eta\big\}
\]
is contained in $\mathcal{E}_N(r,\eta')$, provided $N\ge3$ and 
$\eta'=\eta\cos\frac{\pi}{N}$.
\label{lem-ball:polygon}
\end{lem}

\begin{proof}
This follows from a geometric consideration which involves the inscription of
a regular polygon with $N$ edges inside a circle.
\end{proof}

\begin{cor}
Suppose $g=\Pop\mu$, where $\mu\in L^\infty(\D)$, and consider, for
$0<r<1$ and $0\le\eta<+\infty$, the set
\[
\mathcal{F}(r,\eta):=\big\{\zeta\in\Te:\,\,|g(r\zeta)|
\ge\eta\big\}.
\]
Then the $s$-length of this set enjoys the
bound
\[
|\mathcal{F}(r,\eta)|_s\le \min_{N\ge3}\,
N\exp\Bigg\{-\frac{r^4\eta^2\cos^2\frac{\pi}{N}}
{\|\mu\|_{L^\infty(\D)}^2\log\frac{1}{1-r^2}}\Bigg\}.
\]
\label{cor-weak2}
\end{cor}

\begin{proof}
The assertion is an immediate consequence of Corollary \ref{cor-weak1} together
with Lemma \ref{lem-ball:polygon}. 
\end{proof}

In its turn, this then leads to the following result, which constitutes part
(a) of Theorem \ref{thm-main}.
%has an interpretation as approximately the $L^{1,\infty}(\Te)$ norm 
%of a related function.

\begin{cor}
Suppose $g=\Pop\mu$, where $\mu\in L^\infty(\D)$, and $\|\mu\|_{L^\infty(\D)}\le1$.
Suppose that $0<a<1$.
%, and let $N\ge3$ be an integer so large that $a<\cos^2\frac{\pi}{N}$. 
We
%, for $0<r<1$, 
then have the estimate
\[
%\mathcal{A}(r,\lambda):=\bigg\{\zeta\in\Te:\,\,
%\exp\bigg(\frac{|g(r\zeta)|^2}{\|\mu\|_{L^\infty(\D)}^2
%\log\frac{1}{1-r^2}}\bigg)\ge\lambda\bigg\}.
\int_\Te\exp\Bigg\{a\frac{r^4|g(r\zeta)|^2}
{\log\frac{1}{1-r^2}}\Bigg\}\diff s(\zeta)\le\frac{10}{(1-a)^{3/2}},\qquad
0<r<1.
\]
%Then the $s$-length of this set enjoys the
%bound
%\[
%|\mathcal{A}(r,\lambda)|_s\le \min_{N\ge3}\,
%N\,\lambda^{-r^4\cos^2\frac{\pi}{N}}.
%\]
\label{cor-strong1}
\end{cor}

%\end{document}
\begin{proof}
Let $\nu_r$ be the function defined by
\[
\nu_r(\eta):=|\mathcal{F}(r,\eta)|_s,\qquad
0\le\eta<+\infty,
\]
where the set $\mathcal{F}(r,\eta)$ is as in Lemma \ref{lem-ball:polygon}.
The function $\eta\mapsto\nu_r(\eta)$ is then a decreasing function which takes 
values in the interval $[0,1]$. We realize that
\[
\int_\Te\exp\Bigg\{a\frac{r^4|g(r\zeta)|^2}
{\log\frac{1}{1-r^2}}\Bigg\}\diff s(\zeta)=-\int_0^{+\infty}
\exp\Bigg\{a\frac{r^4\eta^2}
{\log\frac{1}{1-r^2}}\Bigg\}\diff \nu_r(\eta),
\]
and an application of integration by parts together with the estimate of 
Corollary \ref{cor-weak2} (for big enough $N$) shows that
\begin{equation}
\int_\Te\exp\Bigg\{a\frac{r^4|g(r\zeta)|^2}
{\log\frac{1}{1-r^2}}\Bigg\}\diff s(\zeta)=1+\int_0^{+\infty}
\exp\Bigg\{a\frac{r^4\eta^2}
{\log\frac{1}{1-r^2}}\Bigg\}\frac{2a r^4\eta}{\log\frac{1}{1-r^2}}
\nu_r(\eta)\diff\eta.
\label{eq-IBP1.1}
\end{equation}
In this step, we already used that $0<a<1$. Next, we let $N\ge3$ be an integer
so big that $0<a<\cos^2\frac{\pi}{N}$ holds. The estimate of Corollary 
\ref{cor-weak2} applied to \eqref{eq-IBP1.1} leads to
\begin{multline}
\int_\Te\exp\Bigg\{a\frac{r^4|g(r\zeta)|^2}
{\log\frac{1}{1-r^2}}\Bigg\}\diff s(\zeta)
\\
\le1+N\int_0^{+\infty}
\exp\Bigg\{-(\cos^2\tfrac{\pi}{N}-a)\frac{r^4\eta^2}
{\log\frac{1}{1-r^2}}\Bigg\}\frac{2a r^4\eta}{\log\frac{1}{1-r^2}}
\diff\eta=1+\frac{aN}{\cos^2\frac{\pi}{N}-a}.
\label{eq-IBP1.2}
\end{multline}
It remains to choose $N$. We pick $N$ to be \emph{the smallest integer with} 
\[
N\ge \frac{\pi\sqrt{3}}{(1-a)^{1/2}};
\]
then automatically, $N>5$, and 
\[
\cos^2\frac{\pi}{N}-a>1-a-\frac{\pi^2}{N^2}\ge1-a-\frac{1-a}{3}=\frac23(1-a).
\]
At the same time, we have that
\[
N\le1+\frac{\pi\sqrt{3}}{(1-a)^{1/2}}\le\frac{1+\pi\sqrt{3}}{(1-a)^{1/2}},
\]
and a combination with the above estimate shows that
\[
1+\frac{aN}{\cos^2\frac{\pi}{N}-a}\le 1+\frac{\pi\sqrt{3}+1}{2/3}\,
\frac{a}{(1-a)^{3/2}}\le 1+\frac{9a}{(1-a)^{3/2}}\le\frac{10}{(1-a)^{3/2}}.
\]
The assertion of the corollary now follows from the estimate 
\eqref{eq-IBP1.2}.  
\end{proof}

\begin{rem}
In particular, Corollary \ref{cor-strong1} shows that $\mathrm{atvar}\,\Pop\mu
\le\|\mu\|_{L^\infty(\D)}^2$ for functions $\mu\in L^\infty(\D)$.
\end{rem}

\subsection{The control of the moments of a Bloch function}
\label{subsec-moments}

We begin with the following easy lemma.

\begin{lem}
For $0<s<+\infty$, we have the inequality
\[
y^s\le s^s\e^{-s+y},\qquad 0\le y<+\infty.
\]
\end{lem}

The proof is a calculus exercise and therefore omitted.

\begin{proof}[Proof of Corollary \ref{cor-moments}]
Without loss of generality, we may assume $\|\mu\|_{L^\infty(\D)}=1$.
We apply the above lemma with $s=q/2$ and 
\[
y=ar^4\frac{|g(r\zeta)|^2}{\log\frac{1}{1-r^2}},
\]
where $0<a<1$, and obtain
\[
a^{q/2}r^{2q}\frac{|g(r\zeta)|^q}{\big(\log\frac{1}{1-r^2}\big)^{q/2}}
\le (q/2)^{q/2}\e^{-q/2}
\exp\Bigg\{ar^4\frac{|g(r\zeta)|^2}{\log\frac{1}{1-r^2}}\Bigg\}.
\]
After integration along the circle $\Te$ in the $\zeta$ variable, 
we obtain 
\begin{multline*}
\int_\Te|g(r\zeta)|^q\diff s(\zeta)
\le \bigg(\frac{q}{2\e ar^4}\bigg)^{q/2}\bigg(\log\frac{1}{1-r^2}\bigg)^{q/2}
\int_\Te\exp\Bigg\{ar^4\frac{|g(r\zeta)|^2}{\log\frac{1}{1-r^2}}\Bigg\}
\diff s(\zeta)
\\
\le\frac{10}{a^{q/2}(1-a)^{3/2}}\,\bigg(\frac{q}{2\e r^4}\bigg)^{q/2}
\bigg(\log\frac{1}{1-r^2}\bigg)^{q/2},
\end{multline*}
where in the last inequality we implemented the estimate of Theorem 
\ref{thm-main}. We are free to pick $0<a<1$, and with the choice 
$a:=q/(q+3)$, we arrive at
\begin{equation*}
\int_\Te|g(r\zeta)|^q\diff s(\zeta)
\le 10(\e/3)^{3/2}(3+q)^{3/2}\,\bigg(\frac{q}{2\e}\bigg)^{q/2}
\bigg(\frac{1}{r^4}\log\frac{1}{1-r^2}\bigg)^{q/2},
\end{equation*}
which gives the claimed estimate, since $\e<3$. 
\end{proof} 

%\end{document}
\section{Conformal and quasiconformal mapping}

\subsection{Conformal mappings: the standard classes 
$\classS$ and $\Sigma$}
It is a central theme in Conformal Mapping to analyze the local 
dilation/contraction/rotation of the mapping in question. To be more specific,
we introduce the standard class $\classS$ of univalent functions 
$\varphi:\D\to\C$ subject to the normalizations $\varphi(0)=0$ and 
$\varphi'(0)=1$. We consider the function $h_\varphi(z):=\log\varphi'(z)$, 
which may be referred to as the \emph{local complex distortion exponent}. 
A classical estimate of $h_\varphi$ (due to Koebe and Bieberbach) is the
inequality
\begin{equation}
\big|(1-|z|^2)h'_\varphi(z)-2\bar z\big|=\bigg|(1-|z|^2)
\frac{\varphi''(z)}{\varphi'(z)}-2\bar z\bigg|\le4,\qquad z\in\D.
\label{eq-Koebeest1}
\end{equation}
In particular, $h_\varphi$ is in the Bloch space, with seminorm estimate
$\|h_\varphi\|_{\mathcal{B}(\D)}\le6$. On the other hand, \emph{Becker's univalence
criterion} asserts that if $\varphi:\D\to\C$ is a function which is locally 
univalent, that is, $\varphi'(z)\ne0$ for all $z\in\D$, and, in addition,
$\|h_\varphi\|_{\mathcal{B}(\D)}\le1$, then $\varphi$ is necessarily univalent.
Moreover, the bound $1$ which appears here is best possible 
(see \cite{Bec}, \cite{BecPom}). It seems that we are in a situation somewhat
analogous to the what we found for $\Pop L^\infty(\D)$ in Subsection 
\ref{subsec-PLinfty}: the set 
\[
h_\classS:=\{h_\varphi:\,\varphi\in\classS\}
\]
is contained in $6$ times the unit ball of $\mathcal{B}(\D)$, and every 
element $g$ in the unit ball of $\mathcal{B}(\D)$ with $g(0)=0$ is in
$h_\classS$. One minor difference is that we cannot expect $h_\classS$ to share
the properties of a unit ball (convexity etc). 
The behavior of $h_\varphi=\log\varphi'$ may acquire additional boundary growth
if the image domain $\varphi(\D)$ is unbounded, because the derivative 
$\varphi'$ is taken with respect to the Euclidean structure in the image
$\varphi(\D)\subset\C$. To avoid taking such effects into consideration, we
can pass to the univalent function $\psi:\D_e\to\C_\infty$ given
by 
\begin{equation}
\psi(\zeta):=\frac{1}{\varphi(1/\zeta)},\qquad\zeta\in\D_e,
\label{eq-psi1}
\end{equation}
which has $\psi(\zeta)=\zeta+\Ordo(1)$ as $\zeta\to\infty$ and hence is
element of the class $\Sigma$. As for $\psi$, we know that the complement
of the image domain $\psi(\D_e)$ is a compact continuum which does not
divide the plane, contains the origin, and has diameter at most $4$. The 
derivative of $\psi$ evaluated at the point $1/z$ equals
\begin{equation}
\psi'(1/z)=\frac{z^2\varphi'(z)}{[\varphi(z)]^2},\qquad z\in\D,
\label{eq-psiprim}
\end{equation}
which encourages us to replace the study of $h_\varphi$ by the study of
\begin{equation}
g_\varphi(z):=\log\frac{z^2\varphi'(z)}{[\varphi(z)]^2}
=\log\psi'\bigg(\frac{1}{z}\bigg)=h_\psi\bigg(\frac{1}{z}\bigg),
\qquad z\in\D.
\label{eq-gphi}
\end{equation}
The optimal pointwise estimate for the local complex distortion exponent 
$h_\psi(\zeta)=\log\psi'(\zeta)$ is (see \cite{Dur}, p. 123)
\begin{equation*}
|h_\psi(\zeta)|=\big|\log\psi'(\zeta)\big|\le\log
\frac{|\zeta|^2}{|\zeta|^2-1},\qquad \zeta\in\D_e,
\end{equation*}
which in terms of the function $g_\varphi(z)=h_\psi(1/z)$ reads 
\begin{equation}
|g_\varphi(z)|=|h_\psi(1/z)|=
\bigg|\log\frac{z^2\varphi'(z)}{[\varphi(z)]^2}\bigg|\le
\log\frac{1}{1-|z|^2},\qquad z\in\D.
\label{eq-ptwiseS1}
\end{equation}
%which may be compared with the bound \eqref{eq-pointwise1}.

%\end{document}
\subsection{Goluzin's inequality for the class $\Sigma$}
There is an analogue of \eqref{eq-Koebeest1} found by Goluzin in 1943 
(see \cite{Goluzin}, p. 132, as well as \cite{AbuHed}) which applies to 
the class $\Sigma$, but contrary to first expectations, the estimate is not
essentially better than for the class $\classS$. 
Goluzin's inequality, which is sharp pointwise, reads as follows:
\begin{equation}
\bigg|\zeta h_\psi'(\zeta)
%\frac{\psi''(\zeta)}{\psi'(\zeta)}
+\frac{4|\zeta|^2-2}{|\zeta|^2-1}
-\frac{4|\zeta|^2}{|\zeta|^2-1}\frac{E(1/|\zeta|)}{K(1/|\zeta|)}\bigg|
\le\frac{4|\zeta|^2}{|\zeta|^2-1}
\bigg(1-\frac{E(1/|\zeta|)}{K(1/|\zeta|)}\bigg),
\qquad\zeta\in\D_e,
\label{eq-Goluzin0}
\end{equation}
where $h_\psi(\zeta)=\log\psi'(\zeta)$, and $E$ and $K$ denote the 
elliptic integrals
\[
E(s):=\int_0^1\sqrt{\frac{1-s^2t^2}{1-t^2}}\diff t,\qquad 0\le s\le 1,
\]
and
\[
K(s):=\int_0^1\frac{\diff t}{\sqrt{(1-s^2t^2)(1-t^2)}},\qquad 0\le s< 1.
\]
The ratio $E(s)/K(s)$ tends to $0$ as $s\to1^-$, and it is elementary 
to obtain the estimates
\begin{equation*}
1-s^2\le\frac{E(s)}{K(s)}\le1,\qquad 0\le s<1;
%\label{eq-simple1.1}
\end{equation*}
as a consequence, we have that
\begin{equation}
0\le\frac{4|\zeta|^2}{|\zeta|^2-1} 
\bigg(1-\frac{E(1/|\zeta|)}{K(1/|\zeta|)}\bigg)\le \frac{4}{|\zeta|^2-1},
\qquad\zeta\in\D_e,
\label{eq-simple1.2}
\end{equation}
and
\begin{equation}
-\frac{2}{|\zeta|^2-1}\le\frac{4|\zeta|^2-2}{|\zeta|^2-1}
-\frac{4|\zeta|^2}{|\zeta|^2-1}\frac{E(1/|\zeta|)}{K(1/|\zeta|)}
\le\frac{2}{|\zeta|^2-1},\qquad \zeta\in\D_e.
\label{eq-simple1.3}
\end{equation}
%$E(s)\to1$ and $K(s)\to+\infty$ as $s\to1^-$, we have the limit
%\[
%\lim_{|\zeta|\to1^+}\frac{E(1/|\zeta|)}{K(1/|\zeta|)}=0,
%\]
By inserting the estimates \eqref{eq-simple1.2} and \eqref{eq-simple1.3} 
into Goluzin's inequality \eqref{eq-Goluzin0}, we arrive at
\begin{equation}
|\zeta h_\psi'(\zeta)|\le\frac{6}{|\zeta|^2-1},\qquad \zeta\in\D_e,
\label{eq-simple1.4}
\end{equation}
which in terms of the function $g_\varphi(z)=h_\psi(1/z)$ in \eqref{eq-gphi}
reads
\begin{equation}
(1-|z|^2)|g_\varphi'(z)|\le 6|z|,\qquad z\in\D.
\label{eq-simple1.5}
\end{equation}

\begin{prop}
%If, as in the proof of Proposition \ref{prop-2.1.3}, we put
Let $\nu_\varphi\in L^\infty(\D)$ be the function
\[
\nu_\varphi(z):=(1-|z|^2)\frac{g_\varphi'(z)}{z},\qquad z\in\D.
\]
Then 
%\eqref{eq-simple1.5} gives that 
$\|\nu_\varphi\|_{L^\infty(\D)}\le6$, and
%the calculation \eqref{eq-direct1.01} shows that 
$z^2\Pop\nu_\varphi(z)=g_\varphi(z)$.
\label{prop-muphi1.1}
\end{prop}

\begin{proof}
The norm estimate follows from \eqref{eq-simple1.5}, and the equality
$z^2\Pop\nu_\varphi(z)=g_\varphi(z)$ is as in the calculation 
\eqref{eq-direct1.01}, since $g(0)=g'(0)=0$ holds (compare, e.g., with
the estimate \eqref{eq-ptwiseS1}).
\end{proof}

%\end{document}
\subsection{Holomorphic motion, Beltrami equations, and 
quasiconformal extensions}

We use the standard terminology of quasiconformal theory. So, for instance, 
if $\varphi:\Omega_1\to\Omega_2$ is a homeomorphism of complex domains, and if
$k$ is a real parameter with $0\le k<1$, then $\varphi$ is said to be
\emph{$k$-quasiconformal} if it of the Sobolev class 
$W^{1,2}$ locally, and enjoys the dilatation estimate
\[
|\bar\partial_z\varphi(z)|\le k|\partial_z\varphi(z)|,\qquad z\in\Omega_1,
\] 
in the almost-everywhere sense. 
We will also need the notion of \emph{holomorphic motion} (see \cite{MSS}, 
\cite{Slo}, and the recent book \cite{AIM}).

%Before we carry on with holomorphic motion and Beltrami equations, we need
%to see what the strong form of Marshall's Conjecture \ref{conj-Marshall}
%entails for the class $\Sigma$ of normalized conformal mappings on $\D_e$.
We recall that $\psi\in\Sigma$ means that $\psi:\D_e\to\C_\infty$ is univalent
with $\psi(\zeta)=\zeta+\Ordo(1)$ as $\zeta\to\infty$.
Holomorphic motion will allow us to embed such a $\psi$ with a 
quasiconformal extension $\C\to\C$ 
%any fixed dilate $\psi_R(\zeta):=
%R^{-1}\psi(R\zeta)$, $1<R<+\infty$, 
in a chain of conformal 
mappings indexed by a parameter $\lambda\in\D$. The procedure is somewhat 
analogous to the Loewner chain method, but the deformation is based on ideas 
from quasiconformal theory and Beltrami equations, and some aspects even 
rely on methods from Several Complex Variables.   

% but we have a specific
%situation in mind. We pick a conformal mapping $\varphi\in\classS$, and
%let $\psi\in\Sigma$ be given by \eqref{eq-psi1}. We will need to regularize 
%$\varphi$ a little. To this end, we pick an $r$ with $0<r<1$, and consider 
%the normalized dilate $\varphi_r(z):=r^{-1}\varphi(rz)\in\classS$, and 
%the associated $\psi_{r}\in\Sigma$ given by
%\begin{equation}
%\psi_{r}(\zeta):=\frac{1}{\varphi_r(1/\zeta)}=\frac{r}{\varphi(r/\zeta)},\qquad
%\zeta\in\D_e. 
%\label{eq-psi2}
%\end{equation}
%In view of Theorem 12.5.2 in \cite{AIM}, the restriction $\phi_r$ to $\D$ has 
%an extension $\hat\phi_r:\C\to\C$ which is $\frac{1+r}{1-r}$-quasiconformal,
%and hence
%\begin{equation}
%\hat\psi_{r}(\zeta):=\frac{1}{\hat\varphi_r(1/\zeta)},\qquad
%\zeta\in\C_\infty,
%\label{eq-psi3}
%\end{equation}
%is the associated $\frac{1+r}{1-r}$-quasiconformal extension of $\psi_r$ to 
%the extended plane $\C_\infty$. We put
%\begin{equation}
%\mu_{r}(\zeta):=r^{-1}\frac{\bar\partial_\zeta\hat\psi_r(\zeta)}
%{\partial_z\hat\psi_r(\zeta)},\qquad
%\zeta\in\C_\infty,
%\label{eq-psi4}
%\end{equation} 

We begin with a function $\mu\in L^\infty(\D)$ of norm at most $1$, that is,
more formally, we have that $\mu\in L^\infty(\C)$ with
\begin{equation}
|\mu(\zeta)|\le 1_{\D}(\zeta),\qquad\text{a.e.}\,\,\,
\zeta\in\C_\infty.
\label{eq-psi5}
\end{equation}    
Next, we 
%use Theorem 12.5.3 in \cite{AIM}, which tells us that $\hat\psi_r$
obtain the so-called  standard solution $\Psi:\D\times\C\to\C$ 
%$\Psi(r,\zeta)=\hat\psi_r(\zeta)$ and $\Psi(0,\zeta)=\zeta$. 
to the Beltrami equation
\[
\bar\partial_\zeta\Psi(\lambda,\zeta)=\lambda\mu(\zeta)\,
\partial_\zeta\Psi(\lambda,\zeta),\qquad (\lambda,\zeta)\in\D\times\C.
\]
We need to describe in greater detail how to obtain this standard solution 
$\Psi(\lambda,\zeta)$. To this end, we need the Cauchy 
%and
%conjugate Cauchy 
transform,
\begin{equation}
\cauchy\mu(\zeta):=\int_\C\frac{\mu(w)}{\zeta-w}\diff A(w),
%\qquad
%\bar\cauchy\mu(\zeta):=\int_\C\frac{\mu(w)}{\bar z-\bar w}\diff A(w),
\label{eq-cauchy1}
\end{equation}
as well as the Beurling 
%and conjugate Beurling 
transform
\begin{equation}
\Sop\mu(\zeta):=-\,\pv\int_\C\frac{\mu(w)}{(\zeta-w)^2}\diff A(w).
%\qquad
%\bar\Sop\mu(\zeta):=\int_\C\frac{\mu(w)}{(\bar z-\bar w)^2}\diff A(w).
\label{eq-Sop1}
\end{equation}
%If $\Xi_\lambda$ is given by the series expansion
%\begin{equation}
%\Xi_\lambda(\zeta):=-\int_\C\frac{\mu(w)}{(z-w)^2}\diff A(w),\qquad
%\bar\Sop\mu(\zeta):=\int_\C\frac{\mu(w)}{(\bar z-\bar w)^2}\diff A(w).
%\label{eq-Sop1}
%\end{equation}
If $\Mop_\mu$ stands for the multiplication operator $\Mop_\mu f(\zeta):=
\mu(\zeta)f(\zeta)$, then 
%$\Psi(\lambda,\zeta):=C_0(\lambda)+\Psi_0(\lambda,\zeta)$,
%where $C_0(\lambda)$ is constant in $\zeta$ and depends holomorphically on 
%$\lambda\in\D$, and 
\begin{equation}
\Psi(\lambda,\zeta)=\zeta+\lambda\cauchy\mu(\zeta)+
\lambda^2\cauchy\Mop_{\mu}\Sop\mu(\zeta)+
\lambda^3\cauchy\Mop_{\mu}\Sop\Mop_{\mu}\Sop\mu(\zeta)+\cdots. 
\end{equation}
%where 
%we write $\hat\mu_r:=r^{-1}\mu_r$, which is bounded in 
%modulus by $1_\D$ a.e. (cf. \eqref{eq-psi5}).
%The constant $C_0=C_0(\lambda)$ needs to be adjusted to fit the requirements
%that $\Psi(0,\zeta)=\zeta$ and $\Psi(r,\zeta)=\hat\psi_r(\zeta)$. 
%This is of course easy to do, but we would like to also keep the connection
%with the class $\classS$. More precisely, if we put, in analogy with 
%\eqref{eq-psi3},
%\begin{equation}
%\Phi(\lambda,z):=\frac{1}{\Psi(\lambda,1/z)},\qquad(\lambda,z)\in\D\times
%\C_\infty,
%\label{eq-Phi1}
%\end{equation}
%we would then like each $\Phi(\lambda,\cdot)$ to have a restriction 
%to the disk $\D$ which is the class $\classS$. This is the same as asking 
%that $0\notin\Psi(\lambda,\D_e)$, which is satisfied if 
%$0\in\Psi(\lambda,\D)$. The latter
%requirement may be written in the form $-C_0(\lambda)\in\Psi_0(\lambda,\D)$.
%From the assumptions, we know this is so for $\lambda=r$, since 
%$\hat\psi_r$ is given by \eqref{eq-psi3}. So, there exists a point 
%$\zeta_0\in\D$ such that
%$-C_0(r)=\Psi_0(r,\zeta_0)$, and if we put 
%\[
%C_0(\lambda):=-\Psi_0(\lambda,\lambda\zeta_0/r), \qquad \lambda\in\bar\D(0,r),
%\]
%we get a holomorphic choice which meets all the requirements at least for
%$\lambda\in\bar\D(0,r)$. 
As it turns out, for each fixed $\lambda\in\D$, $\Psi(\lambda,\cdot)$ 
is a quasiconformal mapping of $\C_\infty$, which preserves the point at 
infinity, and whose restriction to $\D_e$ is conformal and is in the class 
$\Sigma^{\langle|\lambda|\rangle}$. 
Since $\partial_\zeta\cauchy=\Sop$, the complex derivative of 
$\Psi(\lambda,\cdot)$ equals
\begin{equation}
\partial_\zeta\Psi(\lambda,\zeta)=1+\lambda\Sop\mu(\zeta)+
\lambda^2\Sop\Mop_{\mu}\Sop\mu(\zeta)+
\lambda^3\Sop\Mop_{\mu}\Sop\Mop_{\mu}\Sop\mu(\zeta)+\cdots,
\label{eq-holmot1}
\end{equation}
and if we take the logarithm, the result is
\begin{equation}
H(\lambda,\zeta):=
\log\partial_\zeta\Psi(\lambda,\zeta)=\lambda\hat H_1(\zeta)+
\lambda^2\hat H_2(\zeta)+\lambda^3\hat H_3(\zeta)+\cdots,
%\lambda\Sop\mu(\zeta)+
%\frac12\lambda^2\big(2\Sop\Mop_{\mu}\Sop\mu(\zeta)-
%[\Sop\mu(\zeta)]^2\big)+\cdots,
\label{eq-holmot1.1}
\end{equation}
where 
\begin{equation}
\hat H_1(\zeta)=\Sop\mu(\zeta),\quad 
\hat H_2(\zeta)=\Sop\Mop_{\mu}\Sop\mu(\zeta)-\frac12\,[\Sop\mu(\zeta)]^2,
\ldots,
\label{eq-holmotsuppl1}
\end{equation}
and the power series \eqref{eq-holmot1.1} converges for $\lambda\in\D$ 
(at least for $\zeta\in\D_e$).

%\end{document}

\subsection{Beurling transform formulation of the main theorem}
As the Beurling transform of $\mu$ is connected with
the Bergman projection of $\mu^*(z)=\mu(\bar z)$ via the relation
\begin{equation}
\Sop\mu\bigg(\frac{1}{z}\bigg)=-z^2\Pop\mu^*(z),\qquad z\in\D,
\label{eq-Sop:Pop}
\end{equation}
Theorem \ref{thm-main}(a) has a formulation involving the Beurling 
$\Sop$ in place of the Bergman projection $\Pop$:
\begin{equation}
\int_{\Te}\exp\Bigg\{a\frac{|\Sop\mu(R\zeta)|^2}
{\log\frac{R^2}{R^2-1}}\Bigg\}\diff s(\zeta)\le C(a),\qquad 1<R<+\infty,\,\,\,
0<a<1,
\label{eq-Marshall1.4}
\end{equation}
where $C(a)=10(1-a)^{-3/2}$. Moreover, by Theorem \ref{thm-main}(b),
no such bound is possible for $1<a<+\infty$.
%In particular, if we just include the first term in the series on the 
%left-hand side, using that $\hat H_1(\zeta)=\Sop\mu(\zeta)$ by 
%\eqref{eq-holmotsuppl1}, we must have under \eqref{eq-Marshall1} that
%\begin{equation}
%\limsup_{R\to 1^+}\int_{\Te}\exp\Bigg\{\alpha\frac{|\Sop\mu(R\zeta)|^2}
%{\log\frac{R^2}{R^2-1}}\Bigg\}\diff s(\zeta)<+\infty,
%\label{eq-Marshall1.4}
%\end{equation}
%for $0<\epsilon<1$. 
%\eqref{eq-Marshall1.4} is equivalent to having the bound
%\begin{equation}
%\mathrm{atvar}\,\Pop\mu\le\|\mu\|_{L^\infty(\D)},
%\label{eq-Marshall1.5}
%\end{equation}
%which is much stronger than Makarov's bound 
%$\mathrm{atvar}\,g\le\|g\|_{\mathcal{B}(\D)}$ 
%(cf. Theorem \ref{thm-Makarov1}). 

%\subsection{Statement of the main result}
%
%In support of Marshall's conjecture, we shall obtain \eqref{eq-Marshall1.5}
%and a stronger estimate for McMullen's asymptotic variance:
%
%\begin{thm}
%If 
%\[
%\mathcal{G}=\{\Pop\mu:\,\,\mu\in L^\infty(\D),
%\,\,\|\mu\|_{L^\infty(\D)}\le1\}, 
%\]
%then $\mathrm{avar}_u\,\mathcal{G}<1$ and $\mathrm{atvar}_u\,\mathcal{G}\le1$. 
%\label{thm-main1}
%\end{thm} 
%
%\begin{rem}
%In Example \ref{} below, we supply an explicit function 
%$\mu=\mu_0\in L^\infty(\D)$ with equality in \eqref{eq-Marshall1.5}, which 
%shows that the bound $\mathrm{atvar}_u\,\mathcal{G}\le1$ of 
%Theorem \ref{thm-main1} is optimal.
%In particular, for that $\mu_0$, the observe that that the asymptotic variance
%and the asymptotic tail variance differ: 
%$\mathrm{avar}\,\Pop\mu_0<\mathrm{atvar}\,\Pop\mu_0$.
%\end{rem}

%\end{document}
\subsection{Estimate from above of the universal integral means 
spectrum of conformal mappings with quasiconformal extension}
\label{subsec-proofthmmain2}

We denote by $\Sigma^{\langle k\rangle}$ the collection of all $\psi\in\Sigma$
that have a $k$-quasiconformal extension 
$\tilde\psi:\C_\infty\to\C_\infty$. Via holomorphic motion, any 
$\psi\in\Sigma^{\langle k\rangle}$ is such that for a suitable constant $C_0$,
the function $\psi+C_0$ may be fitted into a standard Beltrami solution family 
$\Psi(\lambda,\cdot)$ at the parameter value $\lambda=k$. The correct value
of the constant $C_0$ is $C_0:=\lim_{\zeta\to\infty}\zeta-\psi(\zeta)$.
%Moreover, for $\lambda\in\D$ close to $0$, the first term in the expansion
%\eqref{eq-holmot1} is dominant, and Theorem \ref{thm-main1} is seen to lead
%to a strong estimate of the integral means spectrum for the class
%$\Sigma^{\langle k\rangle}$, which we denote by $\mathrm{B}(k,t)$:
%\[
%\mathrm{B}(k,t):=\mathrm{B}_{\Sigma^{\langle k\rangle}}(t)=
%\sup_{\psi\in\Sigma^{\langle k\rangle}}\mathrm{B}_{\psi}(t).
%\]
This means that we may focus our attention to standard Beltrami solution 
families $\Psi(\lambda\zeta)$, and think of $k$ as $|\lambda|$. 
By the global estimate \eqref{eq-simple1.4}, which comes from Goluzin's
inequality, we know that the function $H(\lambda,\zeta)$ defined by 
\eqref{eq-holmot1.1} meets
\begin{equation}
|\zeta\partial_\zeta H(\lambda,\zeta)|\le\frac{6}{|\zeta|^2-1},
\qquad\lambda\in\D,\,\,\,\zeta\in\D_e.
\label{eq-Goluzinest2.1}
\end{equation}
In terms of the function
\begin{equation}
G(\lambda,z):=\frac{H(\lambda,1/z)}{\lambda},\qquad 
\lambda,z\in\D\setminus\{0\},
\label{eq-Goluzinest2.2}
\end{equation}
where the singularities at $\lambda=0$ and $z=0$ are both removable, 
the estimate analogous to 
\eqref{eq-Goluzinest2.1} reads (compare with \eqref{eq-ptwiseS1})
\begin{equation}
(1-|z|^2)|\partial_z G(\lambda,z)|\le\frac{6|z|}{|\lambda|},
\qquad\lambda,z\in\D.
\label{eq-Goluzinest2.3}
\end{equation}
The left-hand side of \eqref{eq-Goluzinest2.3} is subharmonic in 
$\lambda\in\D$, so by the maximum principle, we may improve this estimate 
a little:
\begin{equation}
(1-|z|^2)|\partial_z G(\lambda,z)|\le6|z|,
\qquad\lambda,z\in\D.
\label{eq-Goluzinest2.3.1}
\end{equation}
The expansion \eqref{eq-holmot1.1} has an analogue for $G(\lambda,z)$:
\begin{equation}
G(\lambda,z)=\hat G_0(z)+\lambda\hat G_1(z)+\lambda^2\hat G_2(z)+\cdots,
%\lambda\Sop\mu(\zeta)+
%\frac12\lambda^2\big(2\Sop\Mop_{\mu}\Sop\mu(\zeta)-
%[\Sop\mu(\zeta)]^2\big)+\cdots,
\label{eq-Goluzinest2.4}
\end{equation}
where $\hat G_j(z):=\hat H_{j+1}(1/z)$, so that
\begin{equation}
\hat G_0(z)=\Sop\mu(1/z)=-z^2\Pop\mu^*(z),\quad 
\hat G_1(z)=\Sop\Mop_{\mu}\Sop\mu(1/z)-\frac12\,[\Sop\mu(1/z)]^2,
\ldots.
\label{eq-Goluzinest2.5}
\end{equation}
The function $G(\lambda,z)$ enjoys the global growth estimate
\begin{equation}
|G(\lambda,z)|\le \log\frac{1}{1-|z|^2},\qquad \lambda,z\in\D,
\label{eq-Goluzinest2.6}
\end{equation} 
which may be derived from \eqref{eq-simple1.5} by an application of the 
maximum principle (or the Schwarz lemma). Next, let 
$\nu_\lambda\in L^\infty(\D)$ be the function 
\begin{equation*}
\nu_\lambda(z):=(1-|z|^2)\frac{\partial_z G(\lambda,z)}{z},\qquad 
\lambda,z\in\D,
\end{equation*}
which by \eqref{eq-Goluzinest2.3} and Proposition \ref{prop-muphi1.1} has 
$\|\nu_\lambda\|_{L^\infty(\D)}\le6$ and $z^2\Pop\nu_\lambda(z)=G(\lambda,z)$. 
We consider for a moment the function
\begin{equation}
F(\lambda,z):=\frac{G(\lambda,z)-\hat G_0(z)}{\lambda}=
\frac{z^2\Pop\nu_\lambda(z)+z^2\Pop\mu^*(z)}{\lambda}
=\frac{z^2\Pop(\nu_\lambda+\mu^*)(z)}{\lambda},
%\qquad\lambda,z\in\D,
\label{eq-Flambdaz}
\end{equation}
for $\lambda,z\in\D$, which is holomorphic across $\lambda=0$ in view 
of the expansion \eqref{eq-Goluzinest2.4} and the formulae 
\eqref{eq-Goluzinest2.5}.
Since $\|\mu\|_{L^\infty(\D)}\le1$ and $\|\nu_\lambda\|_{L^\infty(\D)}\le6$, we
clearly have that
\begin{equation}
\bigg\|\frac{\nu_\lambda+\mu^*}{\lambda}\bigg\|_{L^\infty(\D)}
\le\frac{7}{|\lambda|},\qquad \lambda\in\D.
\label{eq-Nest.1}
\end{equation}
Let us choose a small number $0<\epsilon<1$, and let $N_{\lambda,\epsilon}$ 
denote the function which equals
\[
N_{\lambda,\epsilon}(z):=\frac{\nu_\lambda+\mu^*}{\lambda},\qquad z\in\D,
\,\,\,1-\epsilon\le|\lambda|<1,
\] 
whereas for $|\lambda|<1-\epsilon$, the function $N_{\lambda,\epsilon}(z)$ is given 
as the Poisson extension to the interior of the boundary values on 
the circle $|\lambda|=1-\epsilon$.
By the maximum principle applied to \eqref{eq-Nest.1}, we find that
\begin{equation}
\|N_{\lambda,\epsilon}\|_{L^\infty(+D)}\le \frac{7}{1-\epsilon},\qquad \lambda\in\D.
\label{eq-Nest.3}
\end{equation} 
As taking the Poisson extension preserves the holomorphic functions, it is a
consequence of \eqref{eq-Flambdaz} that
$F(\lambda,z)=z^2\Pop N_{\lambda,\epsilon}(z)$, and it follows from 
\eqref{eq-Flambdaz} that
\[
G(\lambda,z)=\hat G_0(z)+\lambda F(\lambda,z)=
-z^2\Pop\mu^*(z)+\lambda z^2\Pop N_{\lambda,\epsilon}(z)=z^2\Pop(-\mu^*+\lambda
N_{\lambda,\epsilon})(z).
\]
From the estimate \eqref{eq-Nest.3}, we see that
\[
\|-\mu^*+\lambda N_{\lambda,\epsilon}\|_{L^{\infty}(\D)}\le 
1+\frac{7|\lambda|}{1-\epsilon}, 
\]
and by Theorem \ref{thm-main}, we get that
\[
\mathrm{atvar}\,G(\lambda,\cdot)\le 
\|-\mu^*+\lambda N_{\lambda,\epsilon}\|_{L^{\infty}(\D)}^2\le
\bigg(1+\frac{7|\lambda|}{1-\epsilon}\bigg)^2.
\]
The left hand side does not depend on the choice of $\epsilon$, and we are
free to let $\epsilon\to0^+$:
\[
\mathrm{atvar}\,G(\lambda,\cdot)\le 
(1+7|\lambda|)^2.
\]
Multiplying $G(\lambda,z)$ by the parameter $\lambda$ results in
multiplying the asymptotic tail variance by $|\lambda|^2$: 
\[
\mathrm{atvar}\,\lambda G(\lambda,z)
=|\lambda|^2\mathrm{atvar}\,G(\lambda,\cdot).
\]
Finally, by Corollary \ref{cor-Marshall1.22}, we may estimate the exponential
type spectrum of the function $\exp(\lambda G(\lambda,\cdot))$:
\begin{equation}
\beta_{\lambda G(\lambda,\cdot)}(t)\le \frac{1}{4}|\lambda t|^2(1+7|\lambda|)^2.
\label{eq-etsplambda}
\end{equation}

%\end{document}

\begin{proof}[Proof of Theorem \ref{thm-main2}]
Since 
\[
\lambda G(\lambda,1/\zeta)=H(\lambda,\zeta)
=\log\partial_\zeta\Psi(\lambda,\zeta),
\]
and the exponential type spectrum estimate \eqref{eq-etsplambda} 
is independent of the choice of dilatation coefficient $\mu$
in the unit ball of $L^\infty(\D)$, this completes the proof of Theorem 
\ref{thm-main2}, since by holomorphic motion any 
$\psi\in\Sigma^{\langle k\rangle}$ can be fitted 
into a standard Beltrami solution family $\Psi(\lambda,\cdot)$ for the
parameter value $\lambda=k$ (see, eg., the book \cite{AIM}). 
\end{proof}
 
\begin{rem}
The source of loss of information in the proof of Theorem \ref{thm-main2}
is the fact that in the expansion \eqref{eq-holmot1.1}, we can only effectively 
analyze the first term, $\lambda \hat H_1(\zeta)$. Much deeper understanding 
should result from an analysis of the rest of the terms, individually and
put together. It is, however, known that each coefficient $\hat H_j(\zeta)$ 
is in planar BMO and hence its restriction to $\D_e$ is in the Bloch space
(see Hamilton's paper \cite{Ham}, which builds on work of Reimann).    
\end{rem}

%\begin{rem}
%This estimate is asymptotically as $k\to0$ congruent with the prediction of
%Prause and Smirnov \cite{PS}, which would have that $\mathrm{B}(k,t)\le
%\frac{1}{4}k^2|t|^2$, at least for real $t$, with equality for $|t|\le2/k$.
%For equality to have a chance to hold, one would need not to lose much in 
%the process which gives the estimate. This gives some indication of what 
%properties the approximal optimizing function $\varphi\in\classS$ ought to
%possess. In particular, in Marshall's estimate
%\eqref{eq-Marshall2}, the only way to have minimal loss for as $\alpha\to1^-$
%%for a given function $\varphi\in\classS$ 
%is for 
%\[
%g_\varphi(r\zeta)\approx\frac{\bar t}{2\alpha}\log\frac{1}{1-r^2}
%\] 
%to hold where the mass of the density  
%\[
%\frac{1}{Z(r,\alpha)}\exp\Bigg\{\alpha\frac{|g_\varphi(r\zeta)|^2}
%{\log\frac{1}{1-r^2}}\Bigg\}
%\]
%is substantial on $\Te$. Here, $Z(r,\alpha)$ is a positive normalizing 
%constant so that we get a probability density.  
%
%\end{rem}
%This estimate then leads to an improvement in Smirnov's dimension estimate 
%$D(k)\le 1+k^2$; the improvement refutes an earlier conjecture of Astala that 
%$D(k)=1+k^2$ (see \cite{Sm}).

%\end{document}

\section{The dimension estimate for quasicircles}

\subsection{Pommerenke's Minkowski dimension bound}
\label{subsec-dimbound}

If we combine the estimate $\mathrm{B}(k,t)\le \frac14k^2|t|^2(1+7k)^2$ from
Theorem \ref{thm-main2} with Pommerenke's Minkowski dimension estimate 
\cite{Pom1}, which refines an estimate of Makarov \cite{Mak1.5}, we obtain 
the following.
Let $F(k,t)$ be the quadratic function 
\[
F(k,t):=\frac14 k^2t^2(1+7k)^2-t+1,
\]
where $0<k<1$ and $1<t<2$ are considered. Then, \emph{if $t=t_k$ is a root to
$F(k,t)=0$ with $1<t_k<2$, and if $\partial_t F(k,t)|_{t=t_k}<0$, we have that 
$D_{M,1s}^+(k)\le t_k$}.
\smallskip

\begin{proof}[Proof of Corollary \ref{cor-dim}]
For small $k$, more precisely, $0<k<(\sqrt{15}-1)/14=0.205\ldots$, 
the equation $F(k,t)=0$ has exactly one root in the interval $1<t<2$, and
denote by it by $t=t_k$; explicitly, it is given by
\[
t_k=\frac{2}{1+\sqrt{1-k^2(1+7k)^2}}=1+\frac{k^2}{4}
+\Ordo(k^3),
\]
where the asymptotics is as $k\to0^+$.
Moreover, it is easy to verify that 
$\partial_t F(k,t)|_{t=t_k}<0$, which means that $F(k,t)$ assumes negative 
values to the right of $t=t_k$.
By Pommerenke's estimate (see above), then, it follows that 
\[
D_{M,1s}^+(k)\le t_k=1+\frac{k^2}{4}+\Ordo(k^3),
\]
and, in view of \eqref{eq-dimequal0} and \eqref{eq-dimequal1}, we obtain
\[
D_{M}^+(k')=D_{M,1s}^+\bigg(\frac{2k'}{1+(k')^2}\bigg)\le 1+(k')^2+\Ordo((k')^3),
\]
as claimed.
\end{proof}

\section{On a conjecture of Marshall}

\subsection{Marshall's conjecture}
The following conjecture is from Donald Marshall's notes \cite{Mar1}. Recall
the notation $\mathrm{atvar}\,g$ and $\mathrm{atvar}_u\,\mathcal{G}$ for the 
asymptotic and uniform asymptotic tail variances of a Bloch function $g$ and
a collection of Bloch functions $\mathcal{G}$, respectively (see
\eqref{eq-astvar1} and \eqref{eq-astvar1.1}).

\begin{conj}
{\rm (Marshall)} For every $\varphi\in\classS$, we have that 
$\mathrm{atvar}\,g_\varphi\le1$, where $g_\varphi$ is given by \eqref{eq-gphi}.
Indeed, we should have that $\mathrm{atvar}_u\,g_{\classS}\le1$, where
$g_{\classS}:=\{g_\varphi:\,\,\varphi\in\classS\}$. 
\label{conj-Marshall}
\end{conj}

More explicitly, the first (weaker) part of Marshall's conjecture amounts to 
having
\begin{equation}
\limsup_{r\to1^-}\int_\Te\exp\Bigg\{a\frac{|g_\varphi(r\zeta)|^2}
{\log\frac{1}{1-r^2}}\Bigg\}\diff s(\zeta)<+\infty 
\label{eq-Marshall.1}
\end{equation}
for every fixed $a$, $0<a<1$, and every conformal mapping 
$\varphi\in\classS$. This goes beyond even Theorem \ref{thm-main}, as we 
shall see.

A related approximate exponential quadratic integrability result was 
obtained by Baranov and Hedenmalm (for details, see 
\cite{BarHed}, pp. 20-23). Following ideas developed by Jones and Makarov
\cite{JoMa}, it was derived from the well-known exponential integrability 
of the Marcinkiewicz-Zygmund integral.

\subsection{Marshall's conjecture and standard Beltrami solution 
families}
The strong form of Marshall's Conjecture \ref{conj-Marshall} says that
for every $0<a<1$, we have that
\begin{equation}
\limsup_{R\to1^+}\sup_{\psi\in\Sigma}
\int_\Te\exp\Bigg\{a\frac{|h_\psi(R\zeta)|^2}
{\log\frac{R^2}{R^2-1}}\Bigg\}\diff s(\zeta)<+\infty, 
\label{eq-Marshall..1}
\end{equation}
where $h_\psi(\zeta)=\log\psi'(\zeta)$, as before. We will analyze some 
implications of this conjecture in the setting of a standard Beltrami solution
family $\Psi(\lambda,\zeta)$, as in Subsection \ref{subsec-proofthmmain2}, and
we retain most of the notation from there.
%We want to study the consequences of Marshall's conjecture 
%\eqref{eq-Marshall1} in the setting of the holomorphic family 
%$\Psi(\lambda,\cdot)$ in $\Sigma$, indexed by $\lambda\in\D$. We pick a
%radius $\varrho$ with $1<\varrho<+\infty$, and consider the concentric 
%circle $\lambda=\varrho\omega$, where $\omega\in\Te$, and realize that
In this context, \eqref{eq-Marshall..1} entails that
\begin{equation}
\limsup_{R\to1^+}\sup_{\lambda\in\D}
\int_\Te\exp\Bigg\{a\frac{|H(\lambda,R\zeta)|^2}
{\log\frac{R^2}{R^2-1}}\Bigg\}\diff s(\zeta)<+\infty, 
\label{eq-Marshall..1.2}
\end{equation} 
for every $0<a<1$. We use polar coordinates and write $\lambda=\rho\omega$,
where $0\le\rho<1$ and $\omega\in\Te$.
From the Plancherel identity and the geometric-arithmetic mean inequality 
(or Jensen's inequality), we see that
\begin{multline}
\int_{\Te}\exp\Bigg\{a\sum_{j=1}^{+\infty}
\frac{|\hat H_j(R\zeta)|^2}
{\log\frac{R^2}{R^2-1}}\Bigg\}\diff s(\zeta)
\\
=\lim_{\rho\to1^+}
\int_{\Te}\exp\Bigg\{a\int_{\Te}\frac{|H(\rho\omega,R\zeta)|^2}
{\log\frac{R^2}{R^2-1}}\diff s(\omega)\Bigg\}\diff s(\zeta)
\\
\le\lim_{\rho\to1^+}
\int_{\Te}\int_\Te\exp\Bigg\{a\frac{|H(\rho\omega,R\zeta)|^2}
{\log\frac{R^2}{R^2-1}}\Bigg\}\diff s(\zeta)\diff s(\omega)
\\
\le\sup_{\lambda\in\D}
\int_\Te\exp\Bigg\{a\frac{|H(\lambda,R\zeta)|^2}
{\log\frac{R^2}{R^2-1}}\Bigg\}\diff s(\zeta). 
\label{eq-Marshall..1.3}
\end{multline}
It is now clear that Conjecture \ref{conj-Marshall} amounts to a far-reaching
extension of Theorem \ref{thm-main}(a), since that theorem only involves the
first term in expansion on the left-hand side of \eqref{eq-Marshall..1.3}.

\begin{rem}
In \cite{Mar1}, Marshall explains how under some additional uniformity 
of the constants involved, the implication arrow in Corollary 
\ref{cor-Marshall1.22} may be reversed. See also the paper by Hedenmalm
and Kayumov \cite{HedKay}, p. 2240. Marshall's conjecture
implies the well-known conjectures of Binder, Kraetzer, Brennan, Carleson 
and Jones, and in a sense it may be thought of as equivalent to a strong
form of the most extensive conjecture (that of Binder). 
\end{rem}
%
%In particular, if we just include the first term in the series on the 
%left-hand side, using that $\hat H_1(\zeta)=\Sop\mu(\zeta)$ by 
%\eqref{eq-holmotsuppl1}, we must have under \eqref{eq-Marshall1} that
%\begin{equation}
%\limsup_{R\to 1^+}\int_{\Te}\exp\Bigg\{\alpha\frac{|\Sop\mu(R\zeta)|^2}
%{\log\frac{R^2}{R^2-1}}\Bigg\}\diff s(\zeta)<+\infty,
%\label{eq-Marshall1.4}
%\end{equation}
%for $0<\epsilon<1$. As the Beurling transform of $\mu$ is connected with
%the Bergman projection of $\mu^*(z)=\mu(\bar z)$ via the relation
%\begin{equation*}
%\Sop\mu\bigg(\frac{1}{z}\bigg)=-z^2\Pop\mu^*(z),\qquad z\in\D,
%\end{equation*}
%\eqref{eq-Marshall1.4} is equivalent to having the bound
%\begin{equation}
%\mathrm{atvar}\,\Pop\mu\le\|\mu\|_{L^\infty(\D)},
%\label{eq-Marshall1.5}
%\end{equation}
%which is much stronger than Makarov's bound 
%$\mathrm{atvar}\,g\le\|g\|_{\mathcal{B}(\D)}$ 
%(cf. Theorem \ref{thm-Makarov1}). 

\end{document}